\documentclass[review,onefignum,onetabnum]{siamart171218}

\usepackage{lipsum}
\usepackage{amsfonts}
\usepackage{graphicx}
\usepackage{epstopdf}
\usepackage{algorithmic}
\ifpdf
  \DeclareGraphicsExtensions{.eps,.pdf,.png,.jpg}
\else
  \DeclareGraphicsExtensions{.eps}
\fi

\newsiamremark{hypothesis}{Hypothesis}
\crefname{hypothesis}{Hypothesis}{Hypotheses}
\newsiamthm{claim}{Claim}
\newtheorem{remark}{Remark}
\nolinenumbers

\headers{Simplified error bounds for turning point expansions}{T. M. Dunster, A. Gil, and J. Segura}

\title{Simplified error bounds for turning point expansions}

\author{T. M. Dunster\thanks{Department of Mathematics and Statistics, San Diego State University, 5500 Campanile Drive, San Diego, CA 92182, USA. 
  (\email{mdunster@sdsu.edu}, \url{https://tmdunster.sdsu.edu}).}
\and A. Gil\thanks{Departamento de Matem\'atica Aplicada y CC. de la Computaci\'on, ETSI Caminos, Universidad de Cantabria, 39005-Santander, Spain. 
  (\email{amparo.gil@unican.es}). }
\and J. Segura \thanks{Departamento de Matem\'aticas, Estad\'{\i}stica y Computaci\'on, 
Facultad de Ciencias, 
Universidad de Cantabria, 39005-Santander, Spain. 
  (\email{javier.segura@unican.es}). } }

\usepackage{amsopn}

\makeatletter
\newcommand*{\addFileDependency}[1]{
  \typeout{(#1)}
  \@addtofilelist{#1}
  \IfFileExists{#1}{}{\typeout{No file #1.}}
}
\makeatother

\newcommand*{\myexternaldocument}[1]{%
    \externaldocument{#1}%
    \addFileDependency{#1.tex}%
    \addFileDependency{#1.aux}%
}

\nolinenumbers

\ifpdf
\hypersetup{
  pdftitle={Simplified error bounds for turning point expansions},
  pdfauthor={T. M. Dunster, A. Gil, and J. Segura}
}
\fi

\myexternaldocument{ex_supplement}

\begin{document}

\maketitle

\begin{abstract}
  Recently, the present authors derived new asymptotic expansions for linear differential equations having a simple turning point. These involve Airy functions and slowly varying coefficient functions, and were simpler than previous approximations, in particular being computable to a high degree of accuracy. Here we present explicit error bounds for these expansions which only involve elementary functions, and thereby provide a simplification of the bounds associated with the classical expansions of F. W. J. Olver.
\end{abstract}



\section{Introduction}
In this paper we obtain error bounds for a recent form of asymptotic
expansions for linear differential equations having a simple turning point. The differential equations we study are of the form 
\begin{equation}
d^{2}w/dz^{2}=\left\{u^{2}f(z) +g(z)\right\} w,
\label{eq1}
\end{equation}
where $u$ is a large parameter, real or complex, and $z$ lies in a complex domain which may be unbounded. Many special functions satisfy equations of this form. The functions $f(z)$ and $g(z)$ are meromorphic in a certain domain $Z$ (precisely defined below), and are independent of $u$ (although the latter restriction can often be relaxed without undue difficulty). We further assume that $f(z)$ has no zeros in $Z$ except for a simple zero at $z=z_{0}$, which is the turning point of the equation.

From standard Liouville transformations we have two new variables, namely $\xi$ (for Liouville-Green expansions involving elementary functions) and $\zeta$ (for turning point expansions involving Airy functions). These are given by
\begin{equation}
\xi=\tfrac{2}{3}\zeta ^{3/2}= \int_{z_{0}}^{z} f^{1/2}(t) dt.
\label{eq2}
\end{equation}
We choose the branch here so that $\xi$ is positive when $\zeta$ approaches $0$ through positive values, and by continuity elsewhere. Note $\zeta$ is an analytic function of $z$ at $z=z_{0}$ ($\zeta=0$) whereas $\xi$ has a branch point at the turning point.

With $\zeta$ defined as above and
\begin{equation}
W(u,\zeta) =\zeta ^{-1/4}f^{1/4}(z) w(u,z), 
\label{eq2a}
\end{equation}
the differential equation (\ref{eq1}) is transformed to 
\begin{equation}
d^{2}W/d\zeta ^{2}=\left\{ u^{2}\zeta +\psi(\zeta) \right\} W,
\label{eq3}
\end{equation}
where 
\begin{equation}
\psi (\zeta) =\tfrac{5}{16}\zeta ^{-2}+\zeta \Phi(z),
\label{eq4}
\end{equation}
in which 
\begin{equation}
\Phi(z) =\frac{4f(z) {f}^{\prime \prime }(z) -5{f}^{\prime 2}(z) }{16f^{3}(z) }+\frac{g(z) }{f(z) }. 
\label{eq5}
\end{equation}

The turning point $z=z_{0}$ of (\ref{eq1}) is mapped to the turning point $\zeta =0$ of (\ref{eq3}). If near $z=z_{0}$ the functions $f(z) $
and $f(z) $ have the following Taylor expansions 
\begin{equation}
f(z) =\sum\limits_{n=1}^{\infty }{f}_{n}\left( z-z_{0}\right)
^{n},\ g(z) =\sum\limits_{n=0}^{\infty }{g}_{n}\left(
z-z_{0}\right) ^{n}  
\label{fgtaylor},
\end{equation}
where $f_{1}\neq 0$, then from (\ref{eq2}), (\ref{eq4}) and (\ref{eq5}) we
find that 
\begin{equation}
\lim_{\zeta \rightarrow 0}\psi (\zeta) =\left( f_{1}\right)
^{-8/3}\left\{ \tfrac{9}{16}f_{1}f_{3}-\tfrac{3}{10}\left( f_{2}\right)
^{2}+\left( f_{1}\right) ^{2}g_{0}\right\}. 
\label{psi(0)}
\end{equation}
We define $\psi (0) $ to take this value, hence rendering $\psi(\zeta)$ analytic at the turning point (which otherwise would be a removable singularity).

Following \cite[Chap. 11, Sect. 8.1]{Olver:1997:ASF} we define three sectors 
\begin{equation}
\mathrm{\mathbf{S}}_{j}=\left\{ \zeta :\left\vert {\arg \left(\zeta e^{-2\pi ij/3}\right) }\right\vert \leq {\tfrac{1}{3}}\pi 
\right\} \ \left( {j=0,\pm 1}\right).
\label{eq6a}
\end{equation}
We also define $\mathrm{\mathbf{T}}_{j}$ to be $\mathrm{\mathbf{S}}_{j}$ rotated negatively by the angle $\tfrac{2}{3}\arg(u)$, so that
\begin{equation}
\mathrm{\mathbf{T}}_{j}=\left\{ \zeta :\left\vert {\arg \left( {u^{2/3}\zeta e^{-2\pi ij/3}}\right) }\right\vert \leq {\tfrac{1}{3}}\pi 
\right\} \ \left( {j=0,\pm 1}\right).
\label{eq6}
\end{equation}

Neglecting $\psi (\zeta)$ in (\ref{eq3}) we obtain the
so-called comparison equation $d^{2}W/d\zeta ^{2}={u^{2}\zeta }W$. This has
numerically satisfactory solutions in terms of the Airy function, namely $\mathrm{Ai}_{j}\left( {u^{2/3}\zeta }\right) :=\mathrm{Ai}\left( {u^{2/3}\zeta e}^{-2\pi ij/3}\right)$ ($j=0,\pm 1$). For large $|u|$ these are characterized as being recessive for $\zeta \in \mathrm{\mathbf{T}}_{j}$ and dominant elsewhere.

In \cite{Olver:1964:EBF} and \cite[Chap. 11, Theorem 9.1]{Olver:1997:ASF} Olver obtained three asymptotic solutions to (\ref{eq1}) in the complex plane, of the form 
\begin{multline}
w_{2n+1,j}(u,z) =\left\{ \dfrac{\zeta }{f(z) }\right\}
^{1/4}\left\{ \mathrm{Ai}_{j}\left( {u^{2/3}\zeta }\right)
\sum\limits_{s=0}^{n}\dfrac{{A_{s}(\zeta) }}{u^{2s}}\right. 
\\ 
\left. +\dfrac{\mathrm{Ai}_{j}^{\prime }\left( {u^{2/3}\zeta }\right) }{u^{4/3}
}\sum\limits_{s=0}^{n-1}\dfrac{{B_{s}(\zeta) }}{u^{2s}}
+\varepsilon _{2n+1,j}(u,\zeta) \right\}, 
\label{Olvertp}
\end{multline}
and explicit bounds on the error terms $\varepsilon _{2n+1,j}(u,\zeta) $ were given. However these bounds are quite complicated since they involve the coefficients $A_{s}(\zeta)$ and $B_{s}(\zeta)$ which themselves are hard to compute (due to iterated integration). An added complication is that the bounds involve so-called auxiliary functions for Airy functions (see \cite[Chap. 11, Sect. 8.3]{Olver:1997:ASF}).

In \cite{Dunster:2017:COA} new asymptotic expansions were derived for solutions of (\ref{eq1}) that involved coefficients which are much simpler to evaluate. In this paper we obtain error bounds for these expansions, and these too are much easier to compute than Olver's. Our new bounds only involve explicitly defined coefficients, along with elementary functions, and in particular do not require complicated auxiliary functions or nested integration.

Let us present the main results from \cite{Dunster:2017:COA}. In the following the use of a circumflex (\textasciicircum) is in accord with the notation of this paper, and is used to distinguish certain functions and paths that are defined in terms of $z$ rather than $\xi$.

Firstly we define the set of coefficients 
\begin{equation}
\hat{F}_{1}(z) ={\tfrac{1}{2}}\Phi (z) ,\ \hat{F}_{2}(z) =-{\tfrac{1}{4}}f^{-1/2}(z) {\Phi }^{\prime
}(z),
\label{eq8}
\end{equation}
and 
\begin{equation}
\hat{F}_{s+1}(z) =-\tfrac{1}{2}f^{-1/2}(z) \hat{{F}
}_{s}^{\prime }(z) -\tfrac{1}{2}\sum\limits_{j=1}^{s-1}{\hat{F}
_{j}(z) \hat{F}_{s-j}(z) }\ \left( {s=2,3,4\cdots }
\right).
\label{eq9}
\end{equation}
The odd coefficients appearing in the asymptotic expansions are then given by
\begin{equation}
\hat{E}_{2s+1}(z) =\int {\hat{F}_{2s+1}(z)
f^{1/2}(z) dz}\ \left( {s=0,1,2,\cdots }\right),
\label{eq7}
\end{equation}
where the integration constants must be chosen so that each $\left(
z-z_{0}\right) ^{1/2}\hat{E}_{2s+1}(z) $ is meromorphic
(non-logarithmic) at the turning point. As shown in \cite{Dunster:2017:COA}, the even ones can be determined without any integration, via the formal expansion
\begin{equation}
\sum\limits_{s=1}^{\infty }\dfrac{\hat{E}_{2s}(z) }{u^{2s}}
\sim -\frac{1}{2}\ln \left\{ 1+\sum\limits_{s=0}^{\infty }\dfrac{\hat{F}_{2s+1}(z) }{u^{2s+2}}\right\} +\sum\limits_{s=1}^{\infty }\dfrac{{\alpha }_{2s}}{u^{2s}},
\label{even}
\end{equation}
where each ${\alpha }_{2s}$ is arbitrarily chosen. These too are meromorphic at the turning point. We remark that the coefficients $\hat{F}_{s}(z) $ can be obtained explicitly, along with the even terms $\hat{E}_{2s}(z) $, with each of the odd terms $\hat{E}_{2s+1}(z) $ requiring just one integration of an explicitly determined function, either explicitly with the aid of symbolic software, or by quadrature.

We next define two sequences $\left\{ a_{s}\right\} _{s=1}^{\infty }$ and $\left\{ \tilde{a}_{s}\right\} _{s=1}^{\infty }$ by $a_{1}=a_{2}={\frac{5}{72}}$, $\tilde{a}_{1}=\tilde{a}_{2}=-{\frac{7}{72}}$, with subsequent terms $a_{s}$ and $\tilde{a}_{s}$ ($s=2,3,\cdots $) satisfying the same recursion formula 
\begin{equation}
b_{s+1}=\tfrac{1}{2}\left( {s+1}\right) b_{s}+\tfrac{1}{2}
\sum\limits_{j=1}^{s-1}{b_{j}b_{s-j}}.
\label{arec}
\end{equation}
Then let
\begin{equation}
\mathcal{E}_{s}(z) =\hat{E}_{s}(z) +
(-1)^{s}a_{s}s^{-1}\xi ^{-s},
\label{eq40}
\end{equation}
and
\begin{equation}
\tilde{\mathcal{E}}_{s}(z) =\hat{E}_{s}(z) +(-1) ^{s}\tilde{a}_{s}s^{-1}\xi^{-s},
\label{eq38}
\end{equation}
where $\xi$ is given by (\ref{eq2}).

In \cite{Dunster:2017:COA} it was then shown that there exist solutions of the form
\begin{equation}
w_{j}(u,z) =f^{-1/4}(z) \zeta ^{1/4}\left\{ \mathrm{Ai}_{j}\left(u^{2/3}\zeta \right) A(u,z) +\mathrm{Ai}_{j}^{\prime }\left( u^{2/3}\zeta \right) B(u,z) \right\},
\label{wjs}
\end{equation}
where
\begin{equation}
A(u,z) \sim {\exp \left\{ \sum\limits_{s=1}^{\infty }\dfrac{
\tilde{\mathcal{E}}_{2s}(z) }{u^{2s}}\right\} \cosh \left\{
\sum\limits_{s=0}^{\infty }\dfrac{\tilde{\mathcal{E}}_{2s+1}(z) 
}{u^{2s+1}}\right\} },
\label{Aexp}
\end{equation}
and
\begin{equation}
B(u,z) \sim \frac{1}{u^{1/3}\zeta ^{1/2}}\exp \left\{ 
\sum\limits_{s=1}^{\infty }{\dfrac{\mathcal{E}_{2s}(z) }{u^{2s}}}\right\} \sinh \left\{ \sum\limits_{s=0}^{\infty }\dfrac{\mathcal{E}
_{2s+1}(z) }{u^{2s+1}}\right\} , 
\label{Bexp}
\end{equation}
in certain complex domains, which we describe in detail in \cref{sec2}.

In this paper we truncate the expansions appearing in (\ref{Aexp}) and (\ref{Bexp}) after a finite number of terms, and obtain bounds for the resulting error terms. So rather than the one error term of (\ref{Olvertp}) and its associated complicated bound, we derive separate error bounds for both the $A(u,z)$ and $B(u,z)$ approximations, and this obviates the need for Airy auxiliary functions, since these functions are slowly-varying throughout the asymptotic region of validity. 

We remark error bounds without auxiliary functions were obtained by Boyd in \cite{Boyd:1987:AEF}, but like Olver's expansions (\ref{Olvertp}) his bounds involve the complicated coefficients $A_{s}(\zeta)$ and $B_{s}(\zeta)$, and required successive approximations. They are consequently more complicated and not easy to compute beyond one term in an expansion. In \cite{Dunster:2001:CEF} convergent expansions were derived for the $A(u,z) $ and $B(u,z) $ coefficient functions, but again these are difficult to compute because they also involve coefficients that are hard to evaluate due to iterated integration. In \cite{Dunster:2014:OEB}  asymptotic solutions of (\ref{eq3}) were derived which involved just the Airy function alone (and not its derivative), and where an asymptotic expansion appeared in the argument of this approximant. Error bounds were given, but as in \cite{Boyd:1987:AEF} and \cite[Chap. 11, Theorem 9.1]{Olver:1997:ASF} these are hard to compute.

The importance of explicit error bounds for asymptotic approximations was dem\-onstrated in an expository paper by Olver in \cite{Olver:1980:AAA}. Olver noted how explicit error bounds can provide useful analytical insight into the nature and reliability of the approximations, enable somewhat unsatisfactory concepts such as multiple asymptotic expansions and generalized asymptotic expansions to be avoided, and lead to significant extensions of asymptotic results. 

On the other hand, from a computational point of view, turning point 
uniform asymptotic expansions are important tools which have been considered
for the efficient computation of a good number of special functions. Examples are the algorithms for 
Bessel functions of real argument and
 complex variable of \cite{Amos:1986:A6A} (based on expansions from \cite{Olver:1997:ASF}), the methods for 
modified Bessel functions of imaginary order of \cite{Gil:2004:CSO} (with expansions from \cite{Dunster:1990:BFO}) and the algorithm for parabolic cylynder functions of \cite{Gil:2006:CTR} (see also \cite{Temme:2000:NAA}). 

In the algorithms \cite{Amos:1986:A6A,Gil:2004:CSO,Gil:2006:CTR}, no error bounds are used for establishing the 
accuracy of the uniform expansions; instead this is certified by checking consistency with other methods of evaluation. 
The reason for this lies in the difficulty of computing error bounds for these expansions. The use of error bounds
 for asymptotic expansions in numerical algorithms is in fact very rare, and we only find examples for expansions 
of Poincar\'e type (see for instance \cite{Fabijonas:2004:COC}). In this paper, we develop computable error bounds for 
turning point expansions, thus opening the possibility of using strict error bounds for the numerical computations with turning 
point asymptotics. A related effort in this direction is that of \cite{Wei:2010:EBF}.

The paper is organized as follows. In \Cref{sec2} we use the new results given in \cite{Dunster:2020:LGE} which provide explicit and simple error bounds for Liouville-Green (LG) expansions of exponential form. These rarely-used expansions were used in \cite{Dunster:2017:COA} to obtain (\ref{Aexp}) and (\ref{Bexp}). We apply Dunster's new results to obtain three fundamental LG asymptotic solutions of (\ref{eq3}) complete with error bounds (which are easy to compute). Also in this section we derive an important connection relation between the three solutions. In addition, we obtain similar expansions, with error bounds, for the Airy functions of complex argument that appear in (\ref{wjs}). Both these new connection relations and Airy expansions are used in the subsequent sections, but it is worth remarking that they are interesting and useful in their own right.

The results of \Cref{sec2} are then applied in \Cref{sec3} to obtain the desired error bounds for the expansions (\ref{Aexp}) and (\ref{Bexp}) for $z$ not too close to the turning point. These in turn are used in \Cref{sec4} to obtain error bounds for $z$ lying in a bounded domain which includes the turning point. As in \cite{Dunster:2017:COA}, the method is to express the asymptotic solutions as a Cauchy integral around a simple positively orientated loop surrounding the turning point, and bounding the error along the loop. 

In \Cref{sec5} we illustrate the new results of \Cref{sec3} with an application to Bessel functions of large order. We show how the new simplified expansions and accompanying error bounds can be constructed, how these can then be matched to the exact solutions, and include some numerical examples of the performance of the bounds.

\section{Liouville-Green expansions and connection coefficients}
\label{sec2}

Here we pre\-sent Liouville-Green expansions of exponential form for three numerically satisfactory solutions of (\ref{eq3}), complete with error bounds. To do so we shall employ the new results given in \cite{Dunster:2020:LGE}. We then use these expansions to obtain a connection relation between the three solutions, which will be used in our subsequent error analysis for the expansions (\ref{wjs}) - (\ref{Bexp}).

We begin by defining certain domains. Firstly, we partition each of the sectors in (\ref{eq6}) by $\mathrm{\mathbf{T}}_{j}=\mathrm{\mathbf{T}}
_{j,k}\cup \mathrm{\mathbf{T}}_{j,l}$ ($j,k,l\in \left\{0,1,-1\right\}$, $j\neq k\neq l\neq j$), where $\mathrm{\mathbf{T}}_{j,k}$ is the closed
subsector of angle $\pi /3$ and adjacent to $\mathrm{\mathbf{T}}_{k}$; for
example $\mathrm{\mathbf{T}}_{0,1}=\left\{ \zeta {:0\leq \arg \left( {u^{2/3}
}\zeta \right) \leq {\tfrac{1}{3}}\pi }\right\} $. We denote $T_{j}$ (respectively $T_{j,k}$) to be the region in the $z$ plane corresponding to the sector $\mathrm{\mathbf{T}}_{j}$
(respectively $\mathrm{\mathbf{T}}_{j,k}$) in the $\zeta $ plane. See \Cref{fig:fig1} for some typical regions in the right half plane for the case $z_{0}$ and $u$ positive.

\begin{figure}[htbp]
  \centering
  \includegraphics{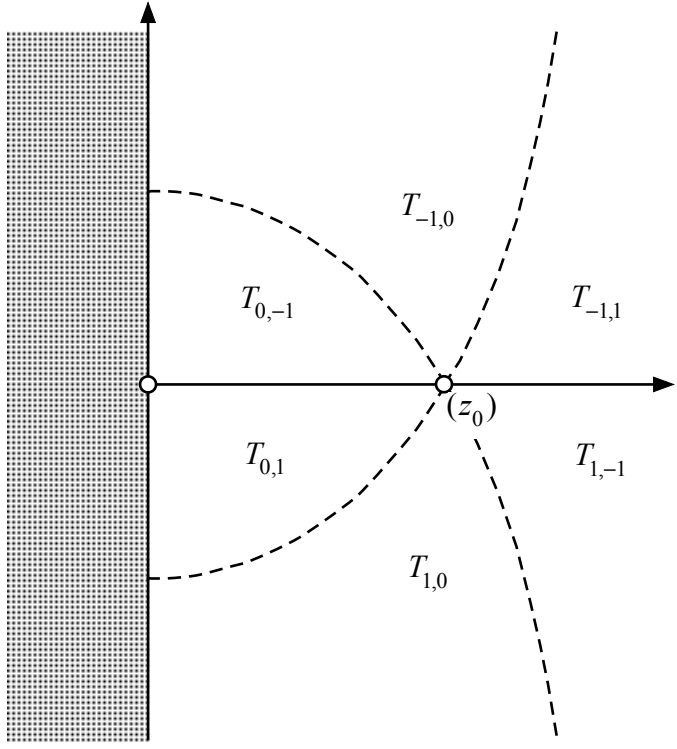}
  \caption{Regions $T_{j,k}$ in $z$ plane for $u$ positive.}
  \label{fig:fig1}
\end{figure}

Next, let $Z$ be the $z$ domain containing $z=z_{0}$ in which $f(z)$ has no other zeros, and in which $f(z)$ and $g(z)$ are meromorphic, with poles (if any) at finite points, at $z=w_{j}$ ($j=1,2,3,\cdots $), say, such that at $z=w_{j}$ (see \cite[Chap. 10, Thm. 4.1]{Olver:1997:ASF}):

(i) $f(z)$ has a pole of order $m>2$, and $g(z)$ is analytic or has a pole of order less than $\frac{1}{2}m+1$, or

(ii) $f(z)$ and $g(z)$ have a double pole, and 
$\left( z-w_{j}\right) ^{2}{g(z) \rightarrow -}\frac{1}{4}$ as $z\rightarrow w_{j}$.

We shall call these\textit{\ admissible poles}. In some applications the
parameter $u$ in (\ref{eq1}), and hence $g(z)$, can be
redefined by a translation to make a pole admissible (which would not be otherwise).

For $j=0,\pm 1$ choose an arbitrary $z^{(j) }\in T_{j}\cap Z$.
These can be chosen at an ordinary point, at an admissible pole, or at infinity if $f(z)$ and $g(z)$ can be expanded in convergent series in a neighborhood of $z=\infty $ of the form
\begin{equation}
f(z)=z^{m}\sum\limits_{s=0}^{\infty }f_{s}z^{-s},\ g(z)=z^{p}\sum\limits_{s=0}^{\infty }g_{s}{z}^{-s},
\label{fginfinity}
\end{equation}
where $f_{0}\neq 0$, $g_{0}\neq 0$, and either $m$ and $p$ are integers such that $m>-2$ and $p<\frac{1}{2}m-1$, or $m=p=-2$ and $g_{0}=-\frac{1}{4}$. For details and generalizations of (\ref{fginfinity}) see \cite[Chap. 10, Sects. 4 and 5]{Olver:1997:ASF}. In this paper we assume that each $z^{(j)}$ is chosen at infinity satisfying the above conditions, or at an admissible pole.

For each\ $j=0,\pm 1$\ the following LG region of validity $Z_{j}(u,z^{(j)})$ (abbreviated $Z_{j}$) then comprises the $z$ point set for which there is a path $\hat{\mathcal{L}}_{j}$ linking $z$ with $z^{(j) }$ in $Z$ and having the properties (i) $\hat{\mathcal{L}}_{j} $ consists of a finite chain of $R_{2}$ arcs (as defined in \cite[Chap. 5, sec. 3.3]{Olver:1997:ASF}), and (ii) as $v$ passes along $\hat{\mathcal{L}}_{j} $ from $z^{(j) }$ to $z$, the real part of $(-1)^{j}u\xi(v)$ is nonincreasing,
where $\xi(v)$is given by (\ref{eq2}) with $z=v$, and with the chosen sign fixed throughout. Following Olver \cite[Chap. 6, sec. 11]{Olver:1997:ASF} these are called \textit{progressive paths}.

Typically one would choose each $z^{(j) }$ to maximize the size of $Z_{j}(u,z^{(j)}) $; for example, if $\theta=\arg(u)$ and the positive sign is chosen in (\ref{eq2}), one might choose $
z^{(j) }$ corresponding to $\xi =\xi ^{(j) }:=\infty
\exp \left\{ -i\theta +ij\pi \right\} $; in this case $z^{(j) }$ would either also be at infinity (provided (\ref{fginfinity})\ holds), or be
an admissible pole.

We now apply \cite{Dunster:2020:LGE} to (\ref{eq3}), and this leads to the following.
\begin{theorem}
\label{thm:2.1}
Let three solutions of (\ref{eq3}) be given by 
\begin{equation}
W_{0}(u,\zeta) =\frac{1}{\zeta ^{1/4}}\exp \left\{ -u\xi
+\sum\limits_{s=1}^{n-1}{(-1) ^{s}\frac{\hat{E}_{s}\left(
z\right) -\hat{E}_{s}\left( {z^{(0) }}\right) }{u^{s}}}
\right\} \left\{ 1+\eta _{n,0}(u,z) \right\},
\label{eq10}
\end{equation}
and 
\begin{equation}
W_{\pm 1}(u,\zeta) =\frac{1}{\zeta ^{1/4}}\exp \left\{ u\xi
+\sum\limits_{s=1}^{n-1}\frac{\hat{E}_{s}(z) -\hat{E}
_{s}\left( {z^{(\pm 1) }}\right) }{u^{s}}\right\} \left\{ {
1+\eta _{n,\pm 1}(u,z) }\right\}, \label{eq11}
\end{equation}
where the root in (\ref{eq10}) is such $\mathrm{Re}(u \xi)>0$ in $T_{0}$ and $\mathrm{Re}(u \xi)<0$ in $T_{-1}\cup T_{1}$; the branch in (\ref{eq11}) for $W_{j}(u,\zeta) $ ($j=\pm 1$) is such $\mathrm{Re}(u \xi)<0$ in $T_{j}$ and $\mathrm{Re}(u \xi)>0$ in $T_{0}\cup T_{-j}$. Then each solution is independent of $n$, and for $z\in Z_{j}\left( {u,}z^{(j)}\right) $ ($j=0,\pm 1$) 
\begin{equation}
\left\vert \eta _{n,j}(u,z) \right\vert \leq |u| ^{-n}\omega _{n,j}(u,z) \exp \left\{ {\left\vert
u\right\vert ^{-1}\varpi _{n,j}(u,z) +|u|
^{-n}\omega _{n,j}(u,z) }\right\}, 
\label{eq12}
\end{equation}
where 
\begin{multline}
\omega _{n,j}(u,z) =2\int_{z^{(j)}}^{z}{\left\vert {
\hat{F}_{n}(t) f^{1/2}(t) dt}\right\vert } \\ 
+\sum\limits_{s=1}^{n-1}\dfrac{1}{{|u| ^{s}}}{
\int_{z^{(j)}}^{z}{\left\vert {\sum\limits_{k=s}^{n-1}{\hat{F}
_{k}(t) \hat{F}_{s+n-k-1}(t) }f^{1/2}\left(
t\right) dt}\right\vert }},
\label{eq13}
\end{multline}
and 
\begin{equation}
\varpi _{n,j}(u,z) =4\sum\limits_{s=0}^{n-2}\frac{1}{{
|u| ^{s}}}{\int_{z^{(j)}}^{z}{\left\vert {\hat{
{F}}_{s+1}(t) f^{1/2}(t) dt}\right\vert }}.
\label{eq14}
\end{equation}
Here the paths of integration are taken along $\hat{\mathcal{L}}_{j}$.
\end{theorem}
\begin{proof}
From the definition (\ref{eq2}) of $\xi$ and letting $Y(u,\xi) =\zeta ^{1/4}W(u,\zeta) $ we transform
(\ref{eq3}) to 
\begin{equation}
d^{2}Y/d{\xi }^{2}=\left\{ {u^{2}+}\Phi (z) \right\} Y,
\label{LGeqn}
\end{equation}
where $\Phi (z) $ is given by (\ref{eq5}). Then we apply \cite[Theorem 1.1]{Dunster:2020:LGE}, in particular (1.17) yields $W_{0}\left( {
u,\zeta }\right) $, and (1.16) yields $W_{\pm 1}(u,\zeta) $
(with different branches of $\xi$ in the $z$ plane, as described above).
The constants $\hat{E}_{s}\left( {z^{(0) }}\right) $ in (\ref{eq10}) were chosen so that
\begin{equation}
\lim_{z\rightarrow {z^{(0) }}}\zeta ^{1/4}e^{u{\xi }}W_{0}(u,\zeta) =1,
\label{eq14a}
\end{equation}
since from (\ref{eq12}) - (\ref{eq14}) $\lim_{z\rightarrow {z^{(0) }}}{\eta _{n,0}(u,z) =0}$, and hence $W_{0}(u,\zeta) $ is independent of $n$. Similarly for the constants $\hat{{E
}}_{s}\left( {z^{(\pm 1) }}\right) $ in (\ref{eq11}) and the resulting independence of $n$ for $W_{\pm 1}(u,\zeta) $.
\end{proof}
\begin{remark}Note all three solutions are analytic and hence single-valued near $\zeta =0$ even though $\xi$ and the coefficients $\hat{E}_{s}(z) $ are not.
\end{remark}

\subsection{Connection coefficients}
We now obtain a connection formula relating the three solutions $W_{j}(u,\zeta) $ ($j=0,\pm 1$). For this, and also throughout this paper, we assume the following.

\begin{hypothesis}
Let each $z^{(j) }\in T_{j}\cap Z_{j}$ ($j=0,\pm 1$) either be at infinity with (\ref{fginfinity})\ holding, or an
admissible pole. Furthermore, assume $z^{(0) }\in Z_{1}\cap Z_{-1}$ and $z^{( \pm 1) }\in Z_{0}\cap Z_{\mp 1}$, i.e. for $j,k=0,\pm 1$ there is a path consisting of a finite chain of $R_{2}$ arcs, linking $z^{(j)}$ with $z^{(k)}$ in $Z$ such as $z$ passes along the path from $z^{(j) }$ to $z^{(k) }$, the real part of $u\xi $ is monotonic.
\label{hyp}
\end{hypothesis}
\begin{lemma}
Under \Cref{hyp} 
\begin{equation}
\lambda _{-1}W_{-1}(u,\zeta) =iW_{0}(u,\zeta)
+\lambda _{1}W_{1}(u,\zeta),  
\label{eq15}
\end{equation}
where (with $\lambda _{0}:=1$) 
\begin{equation}
\lambda _{j}\exp \left\{ -\sum\limits_{s=1}^{n-1}\frac{\hat{E}_{s}\left( 
{z^{(j) }}\right) }{u^{s}}\right\} =\mu _{n}(u)
\left\{ {1+\delta }_{n,j}(u) \right\}  \label{lambda} \: (j=0,\pm 1),
\end{equation}
in which
\begin{equation}
\mu _{n}(u) =\exp \left\{ {-\sum\limits_{s=1}^{n-1}{
(-1)^{s}\frac{\hat{E}_{s}\left( {z^{(0) }}\right) }{u^{s}}}
}\right\},
\label{eq46}
\end{equation}
\begin{equation}
\delta_{n,\pm 1}(u) =\dfrac{\eta _{n,0}
\left( u,z^{( \mp 1) }\right) -\eta _{n,\pm 1}\left({u,z^{( \mp 1) }}\right) }{1+\eta _{n,\pm 1}\left( {u,z^{( \mp 1) }}\right) },
\label{delta}
\end{equation}
and $\delta_{n,0}(u)=0$.
\end{lemma}
\begin{remark} 
From (\ref{eq12}) and (\ref{delta}) we note that ${\delta}_{n,j}(u) =\mathcal{O}\left(u^{-n}\right) $, and hence from (\ref{lambda}) and (\ref{eq46}) 
\begin{equation}
\lambda _{j}\exp \left\{ -\sum\limits_{s=1}^{n-1}\frac{\hat{E}_{s}\left( 
z^{( j) }\right) }{u^{s}}\right\} =\lambda _{k}\exp \left\{ -\sum\limits_{s=1}^{n-1}\frac{\hat{E}_{s}\left( {z^{(k)}}
\right) }{u^{s}}\right\} \left\{ 1+\mathcal{O}\left( \frac{1}{{u^{n}}}
\right) \right\},
\label{eq22}
\end{equation}
for $j,k\in \left\{ 0,1,-1\right\} $.
\end{remark}
\begin{proof}
The result is trivial for $j=0$ since by definition $\lambda _{0}=1$ and ${\delta }_{n,0}(u) =0$. For $j=-1$ let $z\rightarrow z^{(-1)}$ in (\ref{eq15}) (correspondingly $\xi \rightarrow \xi ^{(-1) }$ and $\zeta \rightarrow \zeta ^{(-1) }$). For $W_{0}(u,\zeta) $ and $W_{-1}(u,\zeta) $ we can use (\ref{eq10}) and (\ref{eq11}), and the latter function vanishes
exponentially in the limit. For $W_{1}(u,\zeta) $ we cross a branch cut as $z\rightarrow z^{(-1) }$, and as such in (\ref{eq11}) we have $\xi \rightarrow -\xi ^{(-1) }$, so that $\mathrm{Re}(u \xi) \rightarrow +\infty $. Thus $W_{1}(u,\zeta) $, like $W_{0}(u,\zeta) $, is exponentially large in this limit. As
remarked earlier, $\hat{E}_{2s}(z) $ and $\left(z-z_{0}\right) ^{1/2}\hat{E}_{2s+1}(z) $ are meromorphic in $Z$, and hence single-valued, since they are analytic in that domain except for a pole at the turning point $z=z_{0}$. Thus we have for the coefficients in (\ref{eq11}) for $W_{1}(u,\zeta) $ that $\hat{E}_{2s}(z) \rightarrow \hat{E}_{2s}\left( z^{(-1) }\right) $ and$
\ \hat{E}_{2s+1}(z) \rightarrow -\hat{E}_{2s+1}\left(
z^{(-1) }\right) $ as $z\rightarrow z^{(-1) }$, and in addition $\zeta ^{-1/4}\rightarrow -i\left\{ \zeta ^{(-1)}\right\}^{-1/4}$. We then have from (\ref{eq15}) 
\begin{equation}
\underset{z\rightarrow z^{(-1) }}{\lim }\left\{ \lambda_{1}W_{1}(u,\zeta) +iW_{0}(u,\zeta) \right\} =0,
\label{eq18}
\end{equation}
and hence 
\begin{multline}
\lambda _{1}\exp \left\{ -\sum\limits_{s=1}^{n-1}\dfrac{\hat{E}
_{s}\left( {z^{(1) }}\right) }{u^{s}}\right\} \left\{ {1+\eta_{n,1}\left( {u,z^{(-1) }}\right) }\right\} \\ 
-\exp \left\{ -\sum\limits_{s=1}^{n-1}(-1) ^{s}\dfrac{\hat{E}_{s}\left( {z^{(0) }}\right) }{u^{s}}\right\} \left\{ {1+\eta _{n,0}\left( {u,z^{(-1) }}\right) }\right\} =0.
\label{eq19}
\end{multline}

Similarly letting $z\rightarrow z^{(1) }$ in (\ref{eq15}) yields 
\begin{multline}
\lambda _{-1}\exp \left\{ -\sum\limits_{s=1}^{n-1}\dfrac{\hat{E}
_{s}\left( {z^{(-1) }}\right) }{u^{s}}\right\} \left\{ {
1+\eta _{n,-1}\left( {u,z^{(1)}}\right) }\right\} \\ 
-\exp \left\{ -\sum\limits_{s=1}^{n-1}{(-1) ^{s}\dfrac{\hat{{E
}}_{s}\left( {z^{(0) }}\right) }{u^{s}}}\right\} \left\{ {
1+\eta _{n,0}\left( {u,z^{(1) }}\right) }\right\} =0.
\label{eq21}
\end{multline}
Then (\ref{lambda}) follows from (\ref{eq46}), (\ref{delta}), (\ref{eq19}) and (\ref{eq21}).
\end{proof}

\begin{remark}
The change in integration constants does not affect the
error bounds. Thus for $\eta _{n,j}(u,\xi) $ we can still use (\ref{eq12}).
\end{remark}
\subsection{Airy functions}
We complete this section by presenting similar LG expansions, complete with error bounds, for the Airy functions appearing in (\ref{wjs}). The proof is given in \cref{secA}. We note that the regions of validity of the following asymptotic expansions are not maximal, but they suffice for our purposes.

\begin{theorem}\label{thm:10}
Let $\left\vert {\arg \left( {u^{2/3}\zeta }\right) }\right\vert
\leq {\frac{2}{3}}\pi $ (or equivalently, from (\ref{eq2}), $\left\vert \arg(u \xi) \right\vert \leq \pi $). Then 
\begin{equation}
\mathrm{Ai}\left( {u^{2/3}\zeta }\right) =\frac{1}{2\pi ^{1/2}u^{1/6}\zeta
^{1/4}}\exp \left\{ -u\xi +\sum\limits_{s=1}^{n-1}{(-1) ^{s}
\frac{a_{s}}{su^{s}\xi ^{s}}}\right\} \left\{ {1+\eta _{n}^{(0)
}(u,\xi) }\right\},
\label{eq24}
\end{equation}
and 
\begin{equation}
\mathrm{Ai}^{\prime }\left( {u^{2/3}\zeta }\right) =-\frac{u^{1/6}\zeta ^{1/4}
}{2\pi ^{1/2}}\exp \left\{-u\xi +\sum\limits_{s=1}^{n-1}{(-1)^{s}\frac{\tilde{a}_{s}}{su^{s}\xi ^{s}}}\right\} \left\{ {1+
\tilde{{\eta }}_{n}^{(0) }(u,\xi) }\right\},
\label{eq25}
\end{equation}
where 
\begin{equation}
\left\vert {\eta _{n}^{(0) }(u,\xi) }\right\vert
\leq |u| ^{-n}\gamma _{n}(u,\xi) \exp
\left\{ {|u| ^{-1}\beta _{n}(u,\xi)
+|u| ^{-n}\gamma _{n}(u,\xi) }\right\},
\label{eq26}
\end{equation}
and 
\begin{equation}
\left\vert {\tilde{{\eta }}_{n}^{(0) }(u,\xi) }
\right\vert \leq |u| ^{-n}\tilde{{\gamma }}_{n}(u,\xi) \exp \left\{ |u| ^{-1}\tilde{{\beta }}
_{n}(u,\xi) +|u| ^{-n}\tilde{{\gamma }}
_{n}(u,\xi) \right\},
\label{eq27}
\end{equation}
where 
\begin{equation}
\gamma _{n}(u,\xi) =\frac{2a_{n}\Lambda _{n+1}}{|\xi| ^{n}}+\frac{1}{|u| \left\vert \xi
\right\vert ^{n+1}}\sum\limits_{s=0}^{n-2}\frac{\Lambda _{n+s+2}}{|u \xi| ^{s}}\sum\limits_{k=s+1}^{n-1}{a_{k}a_{s+n-k}},
\label{eq28}
\end{equation}
\begin{equation}
{\beta }_{n}(u,\xi) =\frac{4}{|\xi| }
\sum\limits_{s=0}^{n-2}{\frac{a_{s+1}\Lambda _{s+2}}{\left\vert {u\xi }
\right\vert ^{s}}},
\label{eq29}
\end{equation}
\begin{equation}
\tilde{\gamma}_{n}(u,\xi) =\frac{2\left\vert \tilde{a}
_{n}\right\vert \Lambda _{n+1}}{|\xi| ^{n}}+\frac{1}{
|u| |\xi| ^{n+1}}
\sum\limits_{s=0}^{n-2}{\frac{\Lambda _{n+s+2}}{
|u \xi| ^{s}}\sum\limits_{k=s+1}^{n-1}\tilde{a}{_{k}\tilde{a}_{s+n-k}}},
\label{eq28a}
\end{equation}
\begin{equation}
{\tilde{\beta}}_{n}(u,\xi) =\frac{4}{\left\vert \xi
\right\vert }\sum\limits_{s=0}^{n-2}{\frac{\left\vert \tilde{a}
_{s+1}\right\vert \Lambda _{s+2}}{|u \xi| ^{s}}},
\label{eq29a}
\end{equation}
and
\begin{equation}
\Lambda _{n}=\dfrac{\pi ^{1/2}\Gamma \left( \frac{1}{2}n-\frac{1}{2}\right) 
}{2\Gamma \left( \frac{1}{2}n\right) },
\label{eq93}
\end{equation}

\end{theorem}

On replacing $\zeta$ by $\zeta e^{\mp 2\pi i/3}$ we have the following, assuming the same branches for $\xi(z)$ as in \Cref{thm:2.1}.

\begin{corollary}\label{cor:A}
For $\left\vert {\arg \left( {u^{2/3}\zeta e^{\mp 2\pi i/3}}
\right) }\right\vert \leq {\frac{2}{3}}\pi $ 
\begin{equation}
\mathrm{Ai}_{\pm 1}\left( {u^{2/3}\zeta }\right) =\frac{e^{\pm \pi i/6}}{2\pi
^{1/2}u^{1/6}\zeta ^{1/4}}\exp \left\{ u\xi +\sum\limits_{s=1}^{n-1}\frac{
a_{s}}{su^{s}\xi ^{s}}\right\} \left\{ {1+\eta _{n}^{(\pm 1)
}(u,\xi) }\right\},
\label{eq30}
\end{equation}
and 
\begin{equation}
\mathrm{Ai}_{\pm 1}^{\prime }\left( {u^{2/3}\zeta }\right) =\frac{e^{\pm \pi
i/6}u^{1/6}\zeta ^{1/4}}{2\pi ^{1/2}}\exp \left\{ u\xi
+\sum\limits_{s=1}^{n-1}\frac{\tilde{a}_{s}}{su^{s}\xi ^{s}}\right\}
\left\{ {1+\tilde{\eta }_{n}^{(\pm 1) }(u,\xi) }\right\},
\label{eq31}
\end{equation}
where the error terms are given by $\eta _{n}^{(\pm 1) }(u,\xi) =\eta _{n}^{(0) }\left( {u,\xi e^{\mp \pi i}} \right) $ and $\tilde{{\eta }}_{n}^{(\pm 1) }(u,\xi) =\tilde{{\eta }}_{n}^{(0) }\left( u,\xi e^{\mp \pi i} \right) $, and satisfy the bounds (\ref{eq26}) and (\ref{eq27}),
respectively.
\end{corollary}

\section{Error bounds away from the turning point}
\label{sec3}

The main result is given by \Cref{thm:main1} below. In leading to this, we present some preliminary results.

We begin, following \cite{Dunster:2017:COA}, by defining $A(u,z) $ and $B(u,z) $ by 
\begin{equation}
\frac{1}{2\pi ^{1/2}u^{1/6}}W_{0}(u,\zeta) =\mathrm{Ai}
_{0}\left( {u^{2/3}\zeta }\right) A(u,z) +\mathrm{Ai}
_{0}^{\prime }\left( {u^{2/3}\zeta }\right) B(u,z),
\label{eq34}
\end{equation}
and 
\begin{equation}
\frac{e^{\pi i/6}\lambda _{1}}{2\pi ^{1/2}u^{1/6}}W_{1}
(u,\zeta)
=\mathrm{Ai}_{1}\left( {u^{2/3}\zeta }\right) A(u,z) +
\mathrm{Ai}_{1}^{\prime }\left( {u^{2/3}\zeta }\right) B(u,z).
\label{eq35}
\end{equation}
This leads to the following identity.

\begin{proposition} 
Under \Cref{hyp}
\begin{equation}
\frac{e^{-\pi i/6}\lambda _{-1}}{2\pi ^{1/2}u^{1/6}}W_{-1}
(u,\zeta) =\mathrm{Ai}_{-1}\left( {u^{2/3}\zeta }\right) A(u,z) +
\mathrm{Ai}_{-1}^{\prime }\left( {u^{2/3}\zeta }\right) B(u,z).
\label{eq36}
\end{equation}
\end{proposition}
\begin{proof}
This follows from (\ref{eq15}), (\ref{eq34}), (\ref{eq35})
and the Airy function connection formula (\cite[Chap. 11, Eq. (8.03)]
{Olver:1997:ASF}) 
\begin{equation}
i\mathrm{Ai}\left( {u^{2/3}\zeta }\right) +e^{-\pi i/6}\mathrm{Ai}_{1}\left( 
{u^{2/3}\zeta }\right) =e^{\pi i/6}\mathrm{Ai}_{-1}\left( {u^{2/3}\zeta }
\right).
\label{eq36a}
\end{equation}
\end{proof}

\begin{corollary} 
Let $z\in Z_{j}\cap Z_{k}$ ($j\neq k$). With $\lambda _{0}=1$ 
\begin{multline}
2A(u,z) =\lambda _{j}\exp \left\{ \sum\limits_{s=1}^{n-1}{
\dfrac{\tilde{\mathcal{E}}_{s}(z) -\hat{E}_{s}\left( {z^{(j) }}\right) }{u^{s}}}\right\} \left\{ 1+\eta _{n,j}
(u,z) \right\} \left\{ 1+\tilde{{\eta }}_{n}^{(k) }
(u,\xi)\right\} \\ 
+\lambda _{k}\exp \left\{ \sum\limits_{s=1}^{n-1}{(-1) ^{s}
\dfrac{\tilde{\mathcal{E}}_{s}(z) -\hat{E}_{s}\left( {z^{(k)}}\right) }{u^{s}}}\right\} \left\{ {1+\eta _{n,k}
(u,z) }\right\} \left\{ {1+\tilde{{\eta }}_{n}^{(j) }
(u,\xi) }\right\},
\label{eq37}
\end{multline}
where $j=\pm 1$, $k=0$ for $z\in T_{0,\pm 1}\cup T_{\pm 1,0}$, and $j=\pm 1$
, $k=\mp 1$ for $z\in T_{\pm 1,\mp 1}$. Under the same conditions

\begin{multline}
2u^{1/3}\zeta ^{1/2}B(u,z) \\ =\lambda _{j}\exp \left\{
\sum\limits_{s=1}^{n-1}\dfrac{\mathcal{E}_{s}(z) -\hat{E}
_{s}\left( {z^{(j) }}\right) }{u^{s}}\right\}
\left\{ {1+\eta _{n,j}(u,z) }\right\} \left\{ {1+\eta
_{n}^{(k) }(u,\xi) }\right\}  \\ 
-\lambda _{k}\exp \left\{ \sum\limits_{s=1}^{n-1}{(-1) ^{s}
\dfrac{\mathcal{E}_{s}(z) -\hat{E}_{s}\left( {z^{(k)
}}\right) }{u^{s}}}\right\} \left\{ {1+\eta _{n,k}(u,z) }
\right\} \left\{ {1+\eta _{n}^{(j) }(u,\xi) }
\right\}.
\label{eq39}
\end{multline}

\end{corollary}

\begin{proof}
Let $z\in T_{0,-1}\cup T_{-1,0}$. Solving (\ref{eq34}) and (\ref{eq36}) 
\begin{multline}
A(u,z) =\pi ^{1/2}u^{-1/6}\left\{ {e^{\pi i/6}W_{0}(u,\zeta) \mathrm{Ai}_{-1}^{\prime }\left( {u^{2/3}\zeta }\right) }
\right.  \\ 
\left. {-\lambda _{-1}W_{-1}(u,\zeta) \mathrm{Ai}^{\prime
}\left( {u^{2/3}\zeta }\right) }\right\},
\label{eq42}
\end{multline}
and 
\begin{multline}
B(u,z) =\pi ^{1/2}u^{-1/6}\left\{ {e^{\pi i/6}W_{0}(u,\zeta) \mathrm{Ai}_{-1}^{\prime }\left( {u^{2/3}\zeta }\right) }
\right.  \\ 
\left. {-e^{\pi i/6}W_{0}(u,\zeta) \mathrm{Ai}_{-1}\left( {
u^{2/3}\zeta }\right) }\right\}.
\label{eq43}
\end{multline}
Then use (\ref{eq10}), (\ref{eq11}), (\ref{eq24}), (\ref{eq25}), (\ref{eq30}) and (\ref{eq31}). For the other sectors we repeat this procedure, starting
by solving an appropriate pair of (\ref{eq34}) - (\ref{eq36}) for $A(u,z) $ and $B(u,z) $.
\end{proof}

We define explicit error terms associated with the expansions in our main theorem below. To do so, first let 
\begin{equation}
\eta _{n,j}^{(k) }(u,z) =\eta_{n,j}(u,z) +\eta _{n}^{(k)}(u,\xi) +\eta_{n,j}(u,z) \eta_{n}^{(k) }(u,\xi),
\label{eq44}
\end{equation}
\begin{equation}
\tilde{{\eta }}_{n,j}^{(k) }(u,z) =\eta_{n,j}(u,z) +\tilde{{\eta }}_{n}^{(k) }(u,\xi) +\eta_{n,j}(u,z) \tilde{{\eta }}_{n}^{(k) }(u,\xi),
\label{eq45}
\end{equation}
\begin{equation}
\mathcal{A}_{2m+2}(u,z) =\left[ \mu _{2m+2}(u) 
\right] ^{-1}f^{-1/4}(z) \zeta ^{1/4}A(u,z),
\label{eq47}
\end{equation}
and
\begin{equation}
\mathcal{B}_{2m+2}(u,z) =\left[ \mu _{2m+2}(u) 
\right] ^{-1}f^{-1/4}(z) \zeta ^{1/4}B(u,z),
\label{eq48}
\end{equation}
where $\mu _{n}(u)$ is given by (\ref{eq46}).

Then using (\ref{lambda}), (\ref{eq37}), (\ref{eq45}) with $n=2m+2$, and (\ref{eq47}) we have 
\begin{multline}
2\left\{ \dfrac{f(z) }{\zeta }\right\} ^{1/4}\mathcal{A}
_{2m+2}(u,z) =\exp \left\{ \sum\limits_{s=1}^{2m+1}{\dfrac{
\tilde{\mathcal{E}}_{s}(z) }{u^{s}}}\right\}  \\ 
+\exp \left\{ \sum\limits_{s=1}^{2m+1}{(-1) ^{s}\dfrac{\tilde{
\mathcal{E}}_{s}(z) }{u^{s}}}\right\} +\tilde{\varepsilon}
_{2m+2}(u,z),
\label{eq51a}
\end{multline}
where 
\begin{multline}
\tilde{\varepsilon}_{2m+2}(u,z) =\exp \left\{ 
\sum\limits_{s=1}^{2m+1}\dfrac{\tilde{\mathcal{E}}_{s}(z) }{
u^{s}}\right\} \tilde{e}_{2m+2,j}^{(k) }(u,z)  \\
+\exp \left\{ \sum\limits_{s=1}^{2m+1}{(-1) ^{s}\dfrac{\tilde{
\mathcal{E}}_{s}(z) }{u^{s}}}\right\} \tilde{e}
_{2m+2,k}^{(j) }(u,z),  
\label{eq51b}
\end{multline}
in which
\begin{equation}
\tilde{e}_{n,j}^{(k) }(u,z) =\tilde{\eta}
_{n,j}^{(k) }(u,z) +{\delta }_{n,j}(u)
+\tilde{\eta}_{n,j}^{(k) }(u,z) \delta_{n,j} (u)   
\label{eq51c}
\end{equation}
and $\tilde{\varepsilon}_{2m+2}(u,z) ={\mathcal{O}}\left( u^{-2m-2}\right) $ uniformly for $z\in Z_{j}\cap
Z_{k}$.

Similarly from (\ref{lambda}), (\ref{eq39}), (\ref{eq44}) with $n=2m+2$, and (\ref{eq48})
\begin{multline}
2u^{1/3}f^{1/4}(z) \zeta ^{1/4}\mathcal{B}_{2m+2}(u,z) =\exp \left\{ \sum\limits_{s=1}^{2m+1}\dfrac{\mathcal{E}_{s}\left(
z\right) }{u^{s}}\right\}  \\
-\exp \left\{ \sum\limits_{s=1}^{2m+1}{(-1) ^{s}\dfrac{
\mathcal{E}_{s}(z) }{u^{s}}}\right\} +\varepsilon _{2m+2}
(u,z),  
\label{eq56b}
\end{multline}
where 
\begin{multline}
\varepsilon _{2m+2}(u,z) =\exp \left\{ \sum\limits_{s=1}^{2m+1}{
\dfrac{\mathcal{E}_{s}(z) }{u^{s}}}\right\} e_{2m+2,j}^{\left(
k\right) }(u,z)  \\
-\exp \left\{ \sum\limits_{s=1}^{2m+1}{(-1) ^{s}\dfrac{
\mathcal{E}_{s}(z) }{u^{s}}}\right\} e_{2m+2,k}^{\left(
j\right) }(u,z),  
\label{eq56d}
\end{multline}
with 
\begin{equation}
e_{n,j}^{(k) }(u,z) =\eta _{n,j}^{(k)
}(u,z) +{\delta }_{n,j}(u) +\eta _{n,j}^{\left(
k\right) }(u,z) \delta _{n,j}(u)   
\label{eq56e}
\end{equation}
and $\varepsilon _{2m+2}(u,z) ={\mathcal{O}}\left(
u^{-2m-2}\right) $ uniformly for $z\in Z_{j}\cap Z_{k}$.

In order to simplify our error bounds we shall make use of the following elementary result.

\begin{lemma}
Let $a$, $b$, $c$, and $d$ be real and non-negative. Then if 
\begin{equation}
a\leq c+d+cd,
\label{L1}
\end{equation}
it follows that
\begin{equation}
a+b+ab\leq \left( b+c+d\right) \left\{ 1+\tfrac{1}{2}\left( b+c+d\right)
\right\} ^{2}.
\label{L2}
\end{equation}
\end{lemma}
\begin{proof}
We have from (\ref{L1}) 
\begin{multline}
\left( b+c+d\right) \left\{ 1+\tfrac{1}{2}\left( b+c+d\right) \right\}
^{2}-\left( a+b+ab\right) \\ 
\geq \left( b+c+d\right) \left\{ 1+\tfrac{1}{2}\left( b+c+d\right) \right\}
^{2}-\left\{ b+c+d+cd+b\left( c+d+cd\right) \right\}.
\label{L3}
\end{multline}
On expanding the RHS it is easy to verify that all the negative terms cancel out, and the result follows.
\end{proof}

\begin{remark}If the constants are small and of the same order of magnitude the bound (\ref{L2}) is quite sharp; more precisely, if each constant is $\mathcal{O}
\left( \varepsilon \right)$ where $\varepsilon \rightarrow 0$ then 
\begin{equation}
a+b+ab=\left( b+c+d\right) \left\{ 1+\tfrac{1}{2}\left( b+c+d\right)
\right\} ^{2}+{\mathcal{O}\left( \varepsilon ^{2}\right) }.
\end{equation}
This is easily verifiable by examining the negative terms appearing in the expansion of the RHS of (\ref{L3}).
\end{remark}

Now from (\ref{eq44}) 
\begin{equation}
\left\vert \eta _{n,j}^{(k) }(u,z) \right\vert
\leq \left\vert \eta _{n,j}(u,z) \right\vert +\left\vert \eta
_{n}^{(k) }(u,\xi) \right\vert +\left\vert \eta
_{n,j}(u,z) \right\vert \left\vert \eta _{n}^{(k)
}(u,\xi) \right\vert.
\label{L4}
\end{equation}
Then on identifying the corresponding terms of (\ref{L4}) with those of (\ref{L1}) we deduce from (\ref{L2}) that for $z\in Z_{j}$ ($j=0,\pm 1$) 
\begin{multline}
\left\vert \eta _{n,j}^{(k) }(u,z) \right\vert
+\left\vert {\delta }_{n,j}(u) \right\vert +\left\vert \eta
_{n,j}^{(k) }(u,z) \right\vert \left\vert {\delta }
_{n,j}(u) \right\vert  \\ 
\leq \left( \left\vert {\delta }_{n,j}(u) \right\vert
+\left\vert \eta _{n,j}(u,z) \right\vert +\left\vert \eta
_{n}^{(k) }(u,\xi) \right\vert \right)  \\ 
\times \left\{ 1+\tfrac{1}{2}\left( \left\vert {\delta }_{n,j}\left(
u\right) \right\vert +\left\vert \eta _{n,j}(u,z) \right\vert
+\left\vert \eta _{n}^{(k) }(u,\xi) \right\vert
\right) \right\} ^{2},
\label{eq101}
\end{multline}
and hence from (\ref{eq12}) and (\ref{eq56e})
\begin{equation}
\left\vert e_{n,j}^{(k) }(u,z) \right\vert
\leq {|u| ^{-n}}e_{n,j}(u,z) \left\{ 1+\tfrac{1}{2}{|u| ^{-n}}e_{n,j}(u,z)
\right\} ^{2},
\label{eq102}
\end{equation}
where $e_{n,j}(u,z) $ is given by (\ref{eq100}) below. A similar bound can be established for $\tilde{e}_{n,j}^{(k) }(u,z)$.

Collecting together (\ref{wjs}), (\ref{eq47}) - (\ref{eq56d}), and (\ref{eq102}) we have arrived at the main result of this section:

\begin{theorem}
\label{thm:main1}
Assume \Cref{hyp} and let $z\in Z_{j}\cap Z_{k}$ ($j,k\in \left\{0,1,-1\right\} $, $j\neq k$). Then for each positive integer $m$ there exist three solutions of (\ref{eq1}) of the form 
\begin{equation}
w_{m,l}(u,z) =\mathrm{Ai}_{l}\left( u^{2/3}\zeta \right) 
\mathcal{A}_{2m+2}(u,z) +\mathrm{Ai}_{l}^{\prime }\left(u^{2/3}\zeta\right) \mathcal{B}_{2m+2}(u,z) \ (l=0,\pm 1),
\label{solutions}
\end{equation}
where
\begin{multline}
\mathcal{A}_{2m+2}(u,z) =\left\{ \dfrac{\zeta }{f(z) }\right\} ^{1/4}\exp \left\{ \sum\limits_{s=1}^{m}\dfrac{
\mathcal{\tilde{E}}_{2s}(z) }{u^{2s}}\right\} \cosh \left\{ \sum\limits_{s=0}^{m}\dfrac{\mathcal{\tilde{E}}_{2s+1}(z) }{u^{2s+1}}\right\}  \\ 
+\dfrac{1}{2}\left\{ \dfrac{\zeta }{f(z) }\right\} ^{1/4}\tilde{\varepsilon}_{2m+2}(u,z),
\label{Ascript}
\end{multline}
and
\begin{multline}
\mathcal{B}_{2m+2}(u,z) =\dfrac{1}{u^{1/3}\left\{ \zeta f(z) \right\} ^{1/4}}\exp \left\{ \sum\limits_{s=1}^{m}\dfrac{
\mathcal{E}_{2s}(z) }{u^{2s}}\right\} \sinh \left\{ \sum\limits_{s=0}^{m}\dfrac{\mathcal{E}_{2s+1}(z) }{u^{2s+1}}\right\}  \\ 
+\dfrac{\varepsilon _{2m+2}(u,z) }{2u^{1/3}\left\{ \zeta f(z) \right\} ^{1/4}},
\label{Bscript}
\end{multline}
such that 
\begin{multline}
\left\vert \tilde{\varepsilon}_{2m+2}(u,z) \right\vert \leq \dfrac{1}{{|u| ^{2m+2}}}\exp \left\{\sum\limits_{s=1}^{2m+1}\mathrm{Re}{\dfrac{\mathcal{\tilde{E}}_{s}(z) }{u^{s}}}\right\} \tilde{e}_{2m+2,j}(u,z) \left\{ 1+\dfrac{\tilde{e}_{2m+2,j}(u,z) }{2{|u| ^{2m+2}}}\right\}
^{2} \\
+\dfrac{1}{{|u| ^{2m+2}}}\exp \left\{
\sum\limits_{s=1}^{2m+1}{(-1) ^{s}\mathrm{Re}\dfrac{\mathcal{
\tilde{E}}_{s}(z) }{u^{s}}}\right\} \tilde{e}_{2m+2,k}\left( {
u,z}\right) \left\{ 1+\dfrac{\tilde{e}_{2m+2,k}(u,z) }{2{
|u| ^{2m+2}}}\right\} ^{2},  \label{Aerror}
\end{multline}
in which 
\begin{multline}
\tilde{e}_{n,j}(u,z) ={|u| ^{n}}\left\vert 
{\delta }_{n,j}(u) \right\vert 
+\omega _{n,j}(u,z) \exp \left\{ {|u|
^{-1}\varpi _{n,j}(u,z) +|u| ^{-n}\omega
_{n,j}(u,z) }\right\}  \\
+\tilde{\gamma}_{n}(u,\xi) \exp \left\{ {\left\vert
u\right\vert ^{-1}\tilde{\beta}_{n}(u,\xi) +\left\vert
u\right\vert ^{-n}\tilde{\gamma}_{n}(u,\xi) }\right\},
\label{eq103}
\end{multline}
and 
\begin{multline}
\left\vert \varepsilon _{2m+2}(u,z) \right\vert \leq \dfrac{1}{|u| ^{2m+2}}\exp \left\{ \sum\limits_{s=1}^{2m+1}
\mathrm{Re}{\dfrac{\mathcal{E}_{s}(z) }{u^{s}}}\right\} e_{2m+2,j}(u,z) \left\{ 1+\dfrac{e_{2m+2,j}(u,z) }{2{|u| ^{2m+2}}}\right\} ^{2} \\ +\dfrac{1}{{|u| ^{2m+2}}}\exp \left\{ \sum\limits_{s=1}^{2m+1}{(-1) ^{s}\mathrm{Re}\dfrac{\mathcal{E}_{s}(z) }{u^{s}}}\right\} e_{2m+2,k}(u,z) \left\{1+\dfrac{e_{2m+2,k}(u,z)}{2{|u| ^{2m+2}}}
\right\} ^{2},
\label{Berror}
\end{multline}
where 
\begin{multline}
e_{n,j}(u,z) ={|u| ^{n}}\left\vert {\delta 
}_{n,j}(u) \right\vert  
+\omega _{n,j}(u,z) \exp \left\{ {|u|
^{-1}\varpi _{n,j}(u,z) +|u| ^{-n}\omega
_{n,j}(u,z) }\right\}  \\
+\gamma _{n}(u,\xi) \exp \left\{ {|u|
^{-1}\beta _{n}(u,\xi) +|u| ^{-n}\gamma
_{n}(u,\xi) }\right\}.
\label{eq100}
\end{multline}
In (\ref{Aerror}) and (\ref{Berror}) $j=\pm 1$, $k=0$ for $z\in T_{0,\pm
1}\cup T_{\pm 1,0}$, and $j=\pm 1$, $k=\mp 1$ for $z\in T_{\pm 1,\mp 1}$.
\end{theorem}

\begin{remark}
Here $\mathcal{E}_{s}(z) $ and $\tilde{\mathcal{E}}
_{s}(z) $ are given by (\ref{eq40}) and (\ref{eq38}), $\omega_{n,j}(u,z) $ and $\varpi _{n,j}(u,z)$ are given by (\ref{eq13}) and (\ref{eq14}), $\gamma _{n}(u,\xi)$, $\beta _{n}(u,\xi)$, $\tilde{\gamma}_{n}(u,\xi)$ and $\tilde{\beta}_{n}(u,\xi)$ are given by (\ref{eq28}), (\ref{eq29}), (\ref{eq28a}) and (\ref{eq29a}), ${\delta }_{n,0}(u) =0$, and ${\delta}_{n,\pm 1}(u) $ are bounded using (\ref{delta}); in the common
situation where the connection coefficients $\lambda _{\pm 1}$ of (\ref{eq15}) are known we instead use the exact expressions
\begin{equation}
{\delta }_{n,\pm 1}(u) =\lambda _{\pm 1}\exp \left\{ 
\sum\limits_{s=1}^{n-1}{\frac{{{(-1) ^{s}}}\hat{E}_{s}\left( {z^{(0) }}\right) -\hat{E}_{s}\left( {z^{( \pm 1) }}
\right) }{u^{s}}}\right\} -1.
\label{exactdelta}
\end{equation}
Since $\delta_{n,j}(u)={\mathcal{O}}\left( u^{-n}\right)$ we observe that $\tilde{e}_{n,j}(u,z), e_{n,j}(u,z) ={\mathcal{O}}(1) $ as $u\rightarrow \infty $ uniformly for $z\in Z_{j}$, and hence the bounds for $\tilde{\varepsilon}_{2m+2}(u,z)$ and $\varepsilon _{2m+2}(u,z)$ are both ${\mathcal{O}}\left( u^{-2m-2}\right)$  uniformly for $z\in Z_{j}\cap Z_{k}$.
\end{remark}
\begin{remark}
If the series on the RHS of (\ref{eq56b}) are expanded and combined as an inverse series of $u$ then only (inverse) odd powers remain. Hence one would expect that $\varepsilon _{2m+2}(u,z)={\mathcal{O}}\left( u^{-2m-3}\right)$, and consequently our error bound for the $\mathcal{B}_{2m+2}(u,z)$ expansion  overestimates the true error by a factor ${\mathcal{O}}(u)$. With a more delicate analysis it is possible to sharpen the above bounds to reflect this (and also for the corresponding bounds in \cref{sec4} below). This will be pursued in a subsequent paper.
\label{remark6}
\end{remark}

\section{Error bounds in a vicinity of the turning point}
\label{sec4}

We now consider the case where $z$ is close to $z_{0}$, so that the bounds of the preceding section can no longer be directly applied. As shown in \cite{Dunster:2017:COA} the coefficient functions of (\ref{solutions}) can
be computed to high accuracy by Cauchy integrals in the present case. 

Here we use the same idea to bound the error terms in (\ref{Ascript}) and (\ref{Bscript}). The idea is quite simple: we express the error terms as Cauchy integrals around a simple positively orientated loop $\Gamma $ (say) which encloses the turning point $z_{0}$ and the point $z$ in question (but is not too close to these points), and which lies in the intersection of $Z_{0}$, $Z_{1}$, and $Z_{-1}$. We then bound the integrand of each integral along its contour using the results of the previous section, from which a bound for the error terms follow. The main result is given by \Cref{thm:main2} below

Our choice of $\Gamma $ is the circle $\left\{ z:\left\vert z-z_{0}\right\vert =r_{0}\right\} 
$ for $r_{0}>0$ is arbitrarily chosen but not too small, and such that the loop lies in the intersection of $Z_{0}$, $Z_{1}$, and $Z_{-1}$.

The following result will be used.

\begin{lemma}
For $\left\vert z-z_{0}\right\vert <r_{0}$
\begin{equation}
\oint_{\left\vert t-z_{0}\right\vert =r_{0}}\left\vert {\dfrac{dt}{t-z}}
\right\vert =l_{0}(z) :=\frac{{4r_{0}K}(k) }{
\left\vert z-z_{0}\right\vert +r_{0}},
\label{eq50}
\end{equation}
where
\begin{equation}
k=\frac{{2}\sqrt{r_{0}\left\vert z-z_{0}\right\vert }}{\left\vert
z-z_{0}\right\vert +r_{0}},
\label{defk}
\end{equation}
and $K(k)$ is the complete elliptic integral of the first kind defined by (\cite[\S 19.2(ii)]{NIST:DLMF})
\begin{equation}
{K}(k) =\int\limits_{0}^{\pi /2}{\dfrac{d\tau }{\sqrt{1-k^{2}\sin ^{2}\left( \tau \right) }}}=\int\limits_{0}^{1}{\dfrac{dt}{\sqrt{
\left( 1-t^{2}\right) \left( 1-k^{2}t^{2}\right) }}\ }\left( 0\leq k<1\right).
\label{Kelliptic}
\end{equation}
\end{lemma}

\begin{remark}$K(k) \sim -\frac{1}{2}\ln \left( 1-k\right) $ as $
k\rightarrow 1-$ (\cite[Eq. 19.12.1]{NIST:DLMF}), and hence from (\ref{eq50}) and
(\ref{defk}) we find that $l_{0}(z) \sim -2\ln \left(
r_{0}-\left\vert z-z_{0}\right\vert \right) $ as $\left\vert z-z_{0}\right\vert \rightarrow r_{0}-$; i.e. $l_{0}(z)$ becomes unbounded (logarithmically) as $z$ approaches $\Gamma $ from its interior. This means that $z$ should not be too close to $\Gamma$ in our subsequent error bounds.
\end{remark}

\begin{proof}
Let $z=z_{0}+ae^{i\theta }$ where $a=\left\vert z-z_{0}\right\vert$, and then with the change of variable $t=z_{0}+r_{0}e^{i\varphi }$ we find that 
\begin{equation}
\oint_{\left\vert t-z_{0}\right\vert =r_{0}}\left\vert {\dfrac{dt}{t-z}}
\right\vert =\int\limits_{0}^{2\pi }{\dfrac{
r_{0}d\varphi }{\left\vert r_{0}e^{i\varphi }-ae^{i\theta }\right\vert }=}
\int\limits_{0}^{2\pi }{\dfrac{r_{0}d\varphi }{\left\vert r_{0}e^{i\left(
\varphi -\theta \right) }-a\right\vert }}.
\label{eq51}
\end{equation}
Now let $\varphi \rightarrow \varphi +\theta $, and using $2\pi $
periodicity of the integrand, we get
\begin{equation}
\int\limits_{0}^{2\pi }{\dfrac{r_{0}d\varphi }{\left\vert r_{0}e^{i\left(
\varphi -\theta \right) }-a\right\vert }=}\int\limits_{-\theta }^{2\pi
-\theta }{\dfrac{r_{0}d\varphi }{\left\vert r_{0}e^{i\varphi }-a\right\vert }=}\int\limits_{0}^{2\pi }
\dfrac{r_{0}d\varphi }{\sqrt{r_{0}^{2}-2ar_{0}\cos
\left( \varphi \right) +a^{2}}}.
\label{eq52}
\end{equation}
Then from symmetry of the integrand about $\varphi =\pi $, followed by using the identity $\cos(\varphi) = 1-2 \sin^{2}(\tau) $ where $\tau =\varphi /2,$ we obtain 
\begin{multline}
\int\limits_{0}^{2\pi }{\dfrac{r_{0}d\varphi }{\sqrt{r_{0}^{2}-2ar_{0}\cos
\left( \varphi \right) +a^{2}}}=2}\int\limits_{0}^{\pi }
\dfrac{r_{0}d\varphi }{\sqrt{r_{0}^{2}+2ar_{0}\cos \left( \varphi \right) +a^{2}}}
\\ 
=2\int\limits_{0}^{\pi}\dfrac{r_{0}d\varphi }{\sqrt{\left(a+r_{0}\right) ^{2}-4ar_{0}\sin ^{2}\left( \frac{1}{2}\varphi \right) }}=4
\int\limits_{0}^{\pi /2}{\dfrac{r_{0}d\tau }{\sqrt{\left( a+r_{0}\right)
^{2}-4ar_{0}\sin ^{2}(\tau) }}}.
\label{eq53}
\end{multline}
The result then follows from (\ref{Kelliptic}) - (\ref{eq53}) and recalling
that $a=\left\vert z-z_{0}\right\vert $.
\end{proof}

We now bound terms appearing in \Cref{thm:main1} on $\Gamma $ and on certain paths containing parts of this loop. Firstly, let $\gamma _{j,l}$ be the union of part of the loop $\Gamma $ that lies in $T_{j,l}$ ($j,l\in \left\{ 0,1,-1\right\} ,j\neq l$) with an arbitrarily chosen progressive path in $T_{j}$ connecting $\Gamma $ to $z^{(j)}$ (if possible a straight line). There are six of these paths to consider. See \Cref{fig:fig2,fig:fig3} for examples with $\mathrm{Re}\, z \geq 0$, $u >0$, $z_{0}>0$, $z^{(0)}$ an admissible pole at the origin, and $z^{(1)}$ at infinity.

\begin{figure} [htbp]
  \centering
  \includegraphics{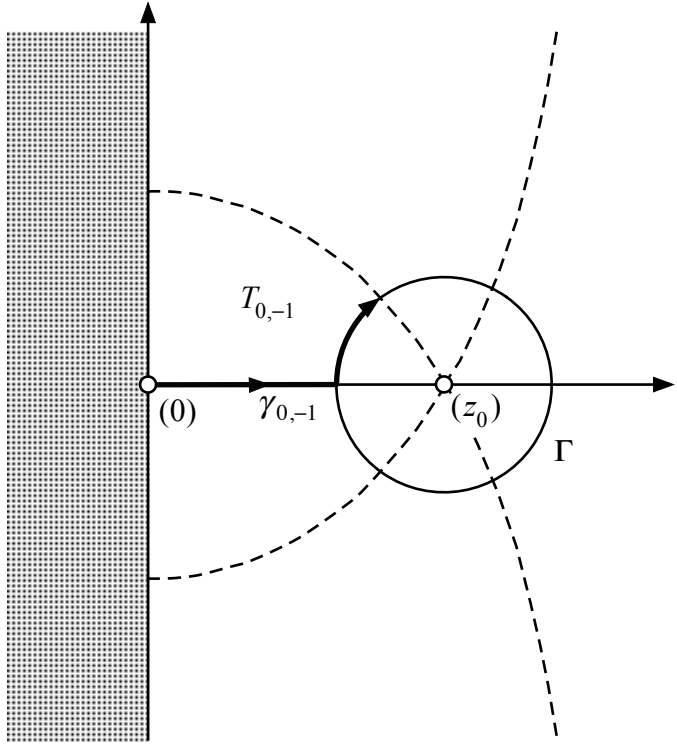}
  \caption{Path $\gamma _{0,-1}$ in the $z$ plane.}
  \label{fig:fig2}
\end{figure}

\begin{figure} [htbp]
  \centering
  \includegraphics{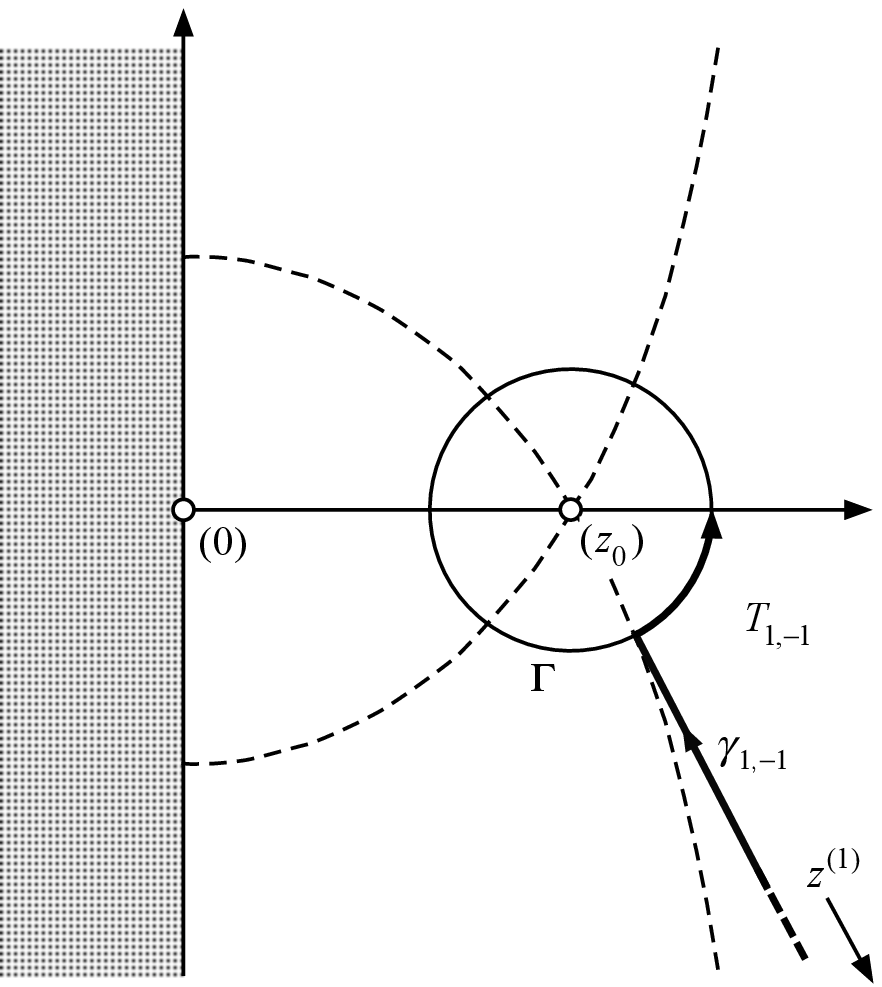}
  \caption{Path $\gamma _{1,-1}$ in the $z$ plane.}
  \label{fig:fig3}
\end{figure}

We then define
\begin{multline}
\omega _{n}(u) =2\max_{j,l} \left\{ \int_{\gamma _{j,l}}{\left\vert{\hat{F}_{n}(t) f^{1/2}(t) dt}\right\vert }\right\} 
\\ 
+\sum\limits_{s=1}^{n-1}\dfrac{1}{{|u| ^{s}}}{
\sum\limits_{k=s}^{n-1}\max_{j,l} }\left\{ {\int_{\gamma _{j,l}}{\left\vert {{\hat{F}_{k}(t) \hat{F}_{s+n-k-1}(t) }f^{1/2}(t) dt}\right\vert }}\right\},
\label{eq68}
\end{multline}
and likewise 
\begin{equation}
\varpi _{n}(u) =4\sum\limits_{s=0}^{n-2}\frac{1}{{\left\vert
u\right\vert ^{s}}}{\max_{j,l} }\left\{ {\int_{\gamma_{j,l}}{\left\vert {\hat{F}
_{s+1}(t) f^{1/2}(t) dt}\right\vert }}\right\},
\label{eq69}
\end{equation}
where the maxima are taken over all six paths $\gamma _{j,l}$. 

We next define
\begin{equation}
{\delta }_{n}(u) =\max_{j=\pm 1}\left\vert {\delta }_{n,j}\left(u\right) \right\vert   \label{deltan},
\end{equation}
\begin{equation}
\Upsilon =\underset{z\in\Gamma}{{\inf }}\left\vert \zeta f(z)
\right\vert ^{1/4},\ \tilde{\Upsilon}=\underset{z\in\Gamma}{\sup }\left\vert
\zeta /f(z) \right\vert ^{1/4},
\label{zeds}
\end{equation}
and
\begin{equation}
\rho =\underset{z\in\Gamma}{\inf }|\xi|.
\label{rho}
\end{equation}

Let $\theta =\arg(u)$ we further define
\begin{equation}
M_{s}=\underset{z\in\Gamma}{{\sup }}\,\mathrm{Re}\left\{ e^{-is\theta }\mathcal{E}
_{s}(z) \right\} ,\ N_{s}=\underset{z\in\Gamma}{{\sup }}\,\mathrm{Re}
\left\{ {(-1) ^{s}}e^{-is\theta }\mathcal{E}_{s}(z) \right\},
\label{eq96}
\end{equation}
and likewise $\tilde{M}_{s}$ and $\tilde{N}_{s}$ where $\mathcal{E}_{s}$ is replaced by $\tilde{\mathcal{E}}_{s}$.

From these definitions we note that on the contour $\Gamma $
\begin{equation}
\omega _{n,j}(u,z) \leq \omega _{n}(u) ,\ \varpi _{n,j}(u,z) \leq \varpi _{n}(u) \ \left( j=0,\pm 1\right),
\label{eq96a}
\end{equation}
\begin{equation}
\left\vert {\exp \left\{ \sum\limits_{s=n}^{n-1}{\dfrac{\mathcal{E}
_{s}(z) }{u^{s}}}\right\} }\right\vert \leq \exp \left\{
\sum\limits_{s=n}^{n-1}{\dfrac{M_{s}}{|u| ^{s}}}
\right\},
\label{eq97}
\end{equation}
and
\begin{equation}
\left\vert {\exp \left\{ \sum\limits_{s=n}^{n-1}{(-1) ^{s}
\dfrac{\mathcal{E}_{s}(z) }{u^{s}}}\right\} }\right\vert \leq
\exp \left\{ \sum\limits_{s=n}^{n-1}{\dfrac{N_{s}}{|u|
^{s}}}\right\}.
\label{eq98}
\end{equation}

Next we define 
\begin{multline}
d_{2m+2}(u) \\ =\left[ \exp \left\{ \sum\limits_{s=1}^{2m+1}{
\dfrac{M_{s}}{|u| ^{s}}}\right\} +\exp \left\{\sum\limits_{s=1}^{2m+1}{\dfrac{N_{s}}{|u| ^{s}}}
\right\} \right]  e_{2m+2}(u) \left\{ 1+\dfrac{e_{2m+2}(u) 
}{2{|u| ^{2m+2}}}\right\} ^{2},  
\label{d2m+1}
\end{multline}
where $e_{n}(u) ={\mathcal{O}}(1) $ as $
u\rightarrow \infty $ and is given by 
\begin{multline}
e_{n}(u) ={|u| ^{n}}{\delta }_{n}\left(
u\right) +\omega _{n}(u) \exp \left\{ {|u|
^{-1}\varpi _{n}(u) +|u| ^{-n}\omega
_{n}(u) }\right\}  \\
+\gamma _{n}(u,\rho) \exp \left\{ {|u|
^{-1}\beta _{n}(u,\rho) +|u| ^{-n}\gamma
_{n}(u,\rho) }\right\} .  
\label{en}
\end{multline}
Recall $\gamma _{n}(u,\xi) $ is given by (\ref{eq28}), and $\beta _{n}(u,\xi)$ is given by (\ref{eq29}).

Similarly we define 
\begin{multline}
\tilde{d}_{2m+2}(u) \\ =\left[ \exp \left\{ \sum
\limits_{s=1}^{2m+1}\dfrac{\tilde{M}_{s}}{|u| ^{s}}
\right\} +\exp \left\{\sum\limits_{s=1}^{2m+1}\dfrac{\tilde{N}_{s}}{|u| ^{s}}\right\} \right]  \tilde{e}_{2m+2}(u) \left\{ 1+\dfrac{\tilde{e}
_{2m+2}(u) }{2{|u| ^{2m+2}}}\right\} ^{2} ,
\label{d2m+2}
\end{multline}
where $\tilde{e}_{n}(u) ={\mathcal{O}}(1) $ as $
u\rightarrow \infty $ and is given by 
\begin{multline}
\tilde{e}_{n}(u) ={|u| ^{n}} {\ \delta }
_{n}(u) +\omega _{n}(u) \exp \left\{ {\left\vert
u\right\vert ^{-1}\varpi _{n}(u) +|u|
^{-n}\omega _{n}(u) }\right\} \\
+\tilde{\gamma}_{n}(u,\rho) \exp \left\{ {\left\vert
u\right\vert ^{-1}\tilde{\beta}_{n}(u,\rho) +\left\vert
u\right\vert ^{-n}\tilde{\gamma}_{n}(u,\rho) }\right\},
\label{entilda}
\end{multline}
in which ${\tilde{\gamma}}_{n}(u,\xi) $ and $\tilde{\beta}_{n}(u,\xi)$ are given by (\ref{eq28a}) and (\ref{eq29a}), respectively.

We now present the main result of this section.
\begin{theorem}
\label{thm:main2}
Assume \Cref{hyp} and let $\Gamma $ be the circle as described at the beginning of this section, with $z$ lying in its interior. Three solutions of (\ref{eq1}) are then given by (\ref{solutions}) where

\begin{multline}
\mathcal{A}_{2m+2}(u,z) =\dfrac{1}{2\pi i}
\oint_{\left\vert t-z_{0}\right\vert =r_{0}}
\exp \left\{ \sum\limits_{s=1}^{m}\dfrac{\tilde{
\mathcal{E}}_{2s}(t) }{u^{2s}}\right\} 
\\ 
\times \cosh 
\left\{ \sum\limits_{s=0}^{m} \dfrac{\tilde{\mathcal{E}}_{2s+1}(t) }{u^{2s+1}}\right\}  
\left\{ \dfrac{\zeta (t) }{f(t) }\right\}
^{1/4}\dfrac{dt}{t-z}+\dfrac{1}{2}\tilde{\kappa}_{2m+2}(u,z),
\label{m02}
\end{multline}
and 
\begin{multline}
\mathcal{B}_{2m+2}(u,z) =\dfrac{1}{2\pi iu^{1/3}}
\oint_{\left\vert t-z_{0}\right\vert =r_{0}}
\exp \left\{ 
\sum\limits_{s=1}^{m}\dfrac{\mathcal{E}_{2s}(t) }{u^{2s}}
\right\}  
\\ 
\times \sinh \left\{\sum\limits_{s=0}^{m} \dfrac{\mathcal{E}_{2s+1}(t) }{u^{2s+1}}\right\} 
\dfrac{dt}{\left\{ \zeta (t) f(t) \right\}^{1/4}
\left( t-z\right)}+\dfrac{\kappa _{2m+2}(u,z) }{2u^{1/3}},
\label{m01}
\end{multline}
such that
\begin{equation}
\left\vert \tilde{\kappa}_{2m+2}(u,z) \right\vert \leq \dfrac{
\tilde{\Upsilon}\tilde{d}_{2m+2}(u) l_{0}(z) }{2\pi
|u| ^{2m+2}},
\label{AkappaBound}
\end{equation}
and
\begin{equation}
\left\vert \kappa _{2m+2}(u,z) \right\vert \leq \dfrac{
d_{2m+2}(u) l_{0}(z) }{2\pi \Upsilon \left\vert
u\right\vert ^{2m+2}}.
\label{BkappaBound}
\end{equation}
\end{theorem}

\begin{proof}
Consider (\ref{m01}). Since $\mathcal{B}_{2m+2}(u,z)$ is analytic on and inside $\Gamma$ we have by Cauchy's integral formula
\begin{equation}
\mathcal{B}_{2m+2}(u,z) =\dfrac{1}{2\pi i}\oint_{\left\vert
t-z_{0}\right\vert =r_{0}}{\dfrac{\mathcal{B}_{2m+2}\left( {u,t}\right) dt}{
\ t-z }}.
  \label{BCauchy}
\end{equation}
On substituting (\ref{eq56b}) into the integrand of (\ref{BCauchy}) and then comparing with (\ref{m01}) we deduce that
\begin{equation}
\kappa _{2m+2}(u,z) =\dfrac{1}{2\pi i}\oint_{\left\vert
t-z_{0}\right\vert =r_{0}}{\dfrac{\varepsilon _{2m+2}(u,t) dt}{
\left\{ \zeta (t) f(t) \right\} ^{1/4}\left(t-z\right) }},
\label{Bkappa}
\end{equation}
(even though $\varepsilon _{2m+2}(u,z)$ is not analytic at the turning point). Therefore from the definition of $\Gamma$ we have from (\ref{eq50}), (\ref{zeds}) and (\ref{Bkappa})
\begin{multline}
\left\vert \kappa _{2m+2}(u,z) \right\vert \leq \dfrac{
\sup_{z\in \Gamma}\left\vert \varepsilon _{2m+2}(u,z) \right\vert }{2\pi \inf_{z\in \Gamma}\left\vert \zeta (z) f(z)
\right\vert ^{1/4}}\oint_{\left\vert t-z_{0}\right\vert =r_{0}}\left\vert {
\dfrac{dt}{t-z}}\right\vert  \\ 
=\dfrac{\sup_{z\in \Gamma}\left\vert \varepsilon _{2m+2}(u,z)
\right\vert l_{0}(z) }{2\pi \Upsilon }.
\label{eq94}
\end{multline}
Now for $z\in\Gamma $ we have from (\ref{rho}) that $|\xi|
\geq \rho $ and hence from (\ref{eq28}) and (\ref{eq29})

\begin{equation}
{\beta _{n}(u,\xi) \leq \beta _{n}(u,\rho) ,\ }
\gamma _{n}(u,\xi) \leq \gamma _{n}(u,\rho).
\label{eq95}
\end{equation}
Thus from (\ref{eq100}), (\ref{deltan}), (\ref{eq96a}) and (\ref{en}) we
have $e_{n,j}(u,z) \leq e_{n}(u) $ for $z\in\Gamma$ and $j=0,\pm 1$. Hence (\ref{BkappaBound}) follows from (\ref{Berror}), (\ref{eq97}), (\ref{eq98}), (\ref{d2m+1}) and (\ref{eq94}). The bound (\ref{AkappaBound}) is similarly proved.
\end{proof}

\section{Bessel functions of large order}
\label{sec5}
We illustrate our new error bounds in an application to Airy expansions for Bessel functions, therefore providing error bounds
for the uniform asymptotic expansions obtained in \cite{Dunster:2017:COA}. For a classical monograph on Bessel functions 
see \cite{Watson:1995:ATO}. 
See also \cite{NIST:DLMF}, chapters 9 and 10, for a compendium of important properties of these functions. 
Similar ideas can be used for bounding the errors
in other cases as for example for Laguerre polynomials and Kummer functions \cite{Dunster:2018:USE}. 

The first step for Bessel functions 
is to apply the Liouville transformations described in \S 1 to Bessel's equation. To this end, we first note that
functions $w=z^{1/2}J_{\nu }(\nu z) $, $w=z^{1/2}H_{\nu
}^{(1) }(\nu z) $ and $w=z^{1/2}H_{\nu }^{\left(
2\right) }(\nu z) $ satisfy 
\begin{equation}
\frac{d^{2}w}{dz^{2}}=\left\{ {\nu ^{2}\frac{1-z^{2}}{z^{2}}-\frac{1}{4z^{2}}
}\right\} w.
\end{equation}
Here $z$ is real or complex, and  $\nu $ plays the role of our parameter $u$, which we assume is real and positive. 
On comparing with (\ref{eq1}) we have 
\begin{equation} f(z)=\frac{1-z^{2}}{z^{2}},\,g(z)=-\frac{1}{4z^{2}}.
\label{fandg}
\end{equation}

For brevity we only consider case \S 3, i.e. $z$ bounded away from the turning point $z_{0}=1$. In a subsequent paper we shall show how our error bounds can be sharpened, including those of \S 4 near the turning point.

The Liouville transformation is 
\begin{equation}
\xi =\frac{2}{3}\zeta ^{3/2}=\ln \left\{ {\frac{1+\left( {1-z^{2}}\right)^{1/2}}{z}}\right\} -\left( {1-z^{2}}\right) ^{1/2},
\label{xiBessel}
\end{equation}
and 
\begin{equation}
W=\left( \frac{1-z^{2}}{\zeta z^2}\right) ^{1/4}w.
\end{equation}
The transformed variable $\zeta $ is real for real $z\in (0,1)$ ($\zeta \in (0,+\infty )$), and $\zeta (z)$ can be defined by analytic continuation in
the whole complex plane cut along the negative real axis. $\xi $ is positive for $z\in (0,1)$ and defined continuously elsewhere.

We then obtain (\ref{eq3}) where 
\begin{equation}
\psi (\zeta) =\frac{5}{16\zeta ^{2}}+\frac{\zeta z^{2}\left( {
z^{2}+4}\right) }{4\left( {z^{2}-1}\right) ^{3}}.
\label{schwarz}
\end{equation}

We find from (\ref{eq5}), (\ref{eq8}) - (\ref{eq7}), and (\ref{fandg}) that the coefficients are given by 
\begin{equation}
\hat{E}_{s}(z) =\int_{z}^{\infty }t^{-1}\left( {1-t^{2}}
\right) ^{1/2}\hat{F}{(t) dt}\quad \left( {s=1,2,3,\cdots }
\right) ,
\end{equation}
Here 
\begin{equation}
\hat{F}_{1}(z)=\frac{{z^{2}(z^{2}+4)}}{{8(z^{2}-1)^{3}}},\,\hat{F}
_{2}(z)=\frac{{z}}{{2}\left( 1-z^{2}\right) ^{1/2}}\hat{F}_{1}^{\prime }(z),
\end{equation}
and
\begin{equation}
\hat{F}_{s+1}(z)=\frac{{z}}{{2}\left( 1-z^{2}\right) ^{1/2}}\hat{F}
_{s}^{\prime }(z)-\frac{1}{2}\sum_{j=1}^{s-1}\hat{F}_{j}(z)\hat{F}_{s-j}(z)\quad \left( {s=2,3,\cdots }\right).  
\label{recuFs}
\end{equation}

As shown in \cite{Dunster:2017:COA} these coefficients can be explicitly computed, and in particular they have the form 
\begin{equation}
\hat{E}_{s}(z)=\frac{{P_{s}(z^{2})}}{{(1-z^{2})^{3s/2}}},  
\label{coebe}
\end{equation}
where $P_{s}(z)$ are polynomials of degree $s$ in $z$.

We note for the odd terms that
\begin{equation}
\hat{E}_{2j+1}(z)=\frac{1}{\left( 1-z\right) ^{1/2}}\left[ \frac{{
P_{2j+1}(z^{2})}}{{(1-z^{2})^{3j+1}}\left( 1+z\right) ^{1/2}}\right] \
\left( j=0,1,2\cdots \right) ,
\end{equation}
where the term in the square brackets is meromorphic at $z=1$ as desired.

The polynomials $P_{s}$ in (\ref{coebe}) have the properties
\begin{equation}
P_{2s}(0)=0,\,P_{2s+1}(0)=C_{2s+1},
\end{equation}
where $C_{2s+1}$ are the coefficients in the Stirling asymptotic series 
\begin{equation}
\Gamma ({\nu })\sim \left( 2\pi \right) ^{1/2}e^{-{\nu }}{\nu }^{{\nu -(1/2)}
}\exp \left\{ \sum_{j=0}^{\infty }\frac{{C_{2j+1}}}{{\nu ^{2j+1}}}\right\}
\ ({\nu }\rightarrow \infty).
\label{stirling}
\end{equation}
Defining $C_{2j}=0$ ($j=1,2,3,\cdots $) we then have 
\begin{equation}
\hat{E}_{s}\left( {z^{(0) }}\right) =\hat{E}_{s}\left( {0}
\right) ={C_{s}},
\label{ez0}
\end{equation}
and from (\ref{coebe}) 
\begin{equation}
\hat{E}_{s}\left( {z^{(\pm 1) }}\right) =\hat{E}_{s}\left( 
{\mp i\infty }\right) =0.
\label{ez1}
\end{equation}

Next, from (\ref{eq10}) and (\ref{eq11}), the following asymptotic solutions are obtained
\begin{equation}
W_{0}(\nu,\zeta) =\frac{1}{\zeta ^{1/4}}\exp \left\{ -\nu
\xi +\sum\limits_{s=1}^{n-1}{(-1) ^{s}\frac{\hat{E}
_{s}(z) -C_{s}}{\nu ^{s}}}\right\} \left\{ 1+\eta _{n,0}(\nu,z) \right\} ,
\end{equation}
and 
\begin{equation}
W_{\pm 1}(\nu,\zeta) =\frac{1}{\zeta ^{1/4}}\exp \left\{ \nu \xi +\sum\limits_{s=1}^{n-1}{\frac{\hat{E}_{s}(z) }{\nu ^{s}}}\right\} \left\{ 1+\eta _{n,\pm 1}(\nu,z) \right\} .
\end{equation}

Let us now match these with the corresponding Bessel functions having the same recessive behavior at the singularities. Firstly, for the one recessive at $z=0$, we note as $z\rightarrow 0$ that
\begin{equation}
J_{\nu }(\nu z) \sim 
\frac{1}{\Gamma(\nu +1) }
\left( \frac{\nu z}{2}\right) ^{\nu } ,
\end{equation}
and hence using
\begin{equation}
\xi = \ln \left( 2/z\right) -1 +\mathcal{O}(z),
\end{equation}
we deduce that 
\begin{equation}
J_{\nu }(\nu z) =\frac{{\nu }^{\nu }}{e^{\nu }\Gamma
(\nu +1) }\left(\frac{\zeta }{1-z^{2}}\right)
^{1/4}W_{0}(\nu,\zeta).
\label{JW0}
\end{equation}

Next, for the solution that vanishes as $z\rightarrow i\infty $,  we use
\begin{equation}
H_{\nu }^{(1) }(\nu z) \sim \left( \frac{2}{\pi {
\nu z}}\right) ^{1/2}\exp \left\{ i{\nu z-}\frac{1}{2}\nu \pi i-\frac{1}{4}{\pi i}\right\} ,
\end{equation}
along with
\begin{equation}
\xi =iz-\tfrac{1}{2}\pi i+\mathcal{O}\left( z^{-1}\right) ,
\label{xi0}
\end{equation}
and we arrive at the identification
\begin{equation}
H_{\nu }^{(1) }(\nu z) =-i\left( \frac{2}{\pi {\nu 
}}\right) ^{1/2}\left( {\frac{\zeta }{1-z^{2}}}\right) ^{1/4}W_{-1}(\nu,\zeta) .
\label{HWm1}
\end{equation}
We similarly find that
\begin{equation}
H_{\nu }^{(2)}(\nu z) =i\left( \frac{2}{\pi {\nu }
}\right) ^{1/2}\left( {\frac{\zeta }{1-z^{2}}}\right) ^{1/4}W_{1}(\nu ,\zeta) .
\end{equation}

We now plug these into the general connection formula (\ref{eq15}), and this yields
\begin{equation}
\lambda _{-1}H_{\nu }^{(1) }(\nu z) =\left( \frac{2
}{\pi {\nu }}\right) ^{1/2}\frac{e^{{\nu }}\Gamma 
(\nu +1) }{{\nu }^{\nu }}J_{\nu }(\nu z) -\lambda _{1}H_{\nu }^{(2) }(\nu z) .
\end{equation}
On comparing this with the well-known connection formula for Bessel functions
\begin{equation}
J_{\nu }(\nu z) =\tfrac{1}{2}\left\{ H_{\nu }^{(1)
}(\nu z) +H_{\nu }^{(2) }(\nu z)
\right\} ,
\end{equation}
we deduce that
\begin{equation}
\lambda _{1}=\lambda _{-1}=\left( \frac{1}{2\pi {\nu }}\right) ^{1/2}\frac{
e^{{\nu }}\Gamma (\nu +1) }{{\nu }^{{\nu }}}.
\label{lambdaBessel}
\end{equation}

We note from (\ref{stirling}) that
\begin{equation}
\lambda _{\pm 1}\sim \exp \left\{ \sum\limits_{j=0}^{\infty }\frac{{
C_{2j+1}}}{{\nu }^{2j+1}}\right\} \\ \left( {\nu }\rightarrow \infty \right) ,
\end{equation}
in accord with (\ref{eq22}), (\ref{ez0}) and (\ref{ez1}).

For $z\in T_{0,-1}\cup T_{-1,0}$ (see \Cref{fig:fig1}) we use (\ref{eq42}), (\ref{eq43}), (\ref{JW0}), (\ref{HWm1}) and (\ref{lambdaBessel}) to obtain the exact expressions
\begin{multline}
A(\nu,z) =\frac{\pi ^{1/2}
e^{{\nu }}\Gamma (\nu +1) }{{\nu }^{{\nu+(1/6) }}}{\left( {\dfrac{1-z^{2}}{\zeta }}\right) ^{1/4}} \\ 
\times \left\{ e^{\pi i/6}\mathrm{Ai}_{-1}^{\prime }\left( {\nu ^{2/3}\zeta }
\right) J_{\nu }(\nu z) -\tfrac{1}{2}i\mathrm{Ai}^{\prime
}\left( {\nu ^{2/3}\zeta }\right) H_{\nu }^{(1) }
(\nu z) \right\} ,
\end{multline}

and
\begin{multline}
B(\nu,z) = \frac{\pi ^{1/2}
e^{{\nu }}\Gamma (\nu +1) }{{\nu }^{{\nu+(1/6) }}}{\left( {\dfrac{1-z^{2}}{\zeta }}\right) ^{1/4}} \\ 
\times \left\{ \tfrac{1}{2}i{\mathrm{Ai}\left( {\nu ^{2/3}\zeta }\right) }
H_{\nu }^{(1) }(\nu z) {-e^{\pi i/6}\mathrm{Ai}
_{-1}\left( {\nu ^{2/3}\zeta }\right) }J_{\nu }(\nu z)
\right\} ,
\end{multline}
Now from (\ref{eq46}), (\ref{eq47}), (\ref{eq48}), (\ref{fandg}) and (\ref{ez0}) we have
\begin{equation}
\mathcal{A}_{2m+2}(\nu,z) =\exp \left\{ -\sum
\limits_{j=0}^{m}\frac{{C_{2j+1}}}{{\nu }^{2j+1}}\right\} \left( \frac{
z^{2}\zeta }{1-z^{2}}\right) ^{1/4}A(\nu,z) ,
\end{equation}
and
\begin{equation}
\mathcal{B}_{2m+2}(\nu,z) =\exp \left\{ -\sum
\limits_{j=0}^{m}\frac{C_{2j+1}}{{\nu }^{2j+1}}\right\} \left( \frac{z^{2}\zeta }{1-z^{2}}\right) ^{1/4}B(\nu,z) ,
\end{equation}
and hence
\begin{multline}
\mathcal{A}_{2m+2}(\nu,z) =\pi ^{1/2}{e^{{\nu }}{\nu }^{-{\nu
+(5/6)}}\Gamma (\nu) }\exp \left\{ -\sum\limits_{j=0}^{m}{
\dfrac{{C_{2j+1}}}{{\nu }^{2j+1}}}\right\}  \\ 
\times z^{1/2}\left\{ e^{\pi i/6}\mathrm{Ai}_{-1}^{\prime }\left( {\nu^{2/3}\zeta }\right) J_{\nu }(\nu z) -\tfrac{1}{2}i\mathrm{Ai}^{\prime }\left( {\nu ^{2/3}\zeta }\right) H_{\nu }^{(1)
}(\nu z) \right\} ,
\label{eq5.32}
\end{multline}
and
\begin{multline}
\mathcal{B}_{2m+2}(\nu,z) =\pi ^{1/2}e^{{\nu }}{{\nu }^{-{\nu
+(5/6)}}}\Gamma (\nu) \exp \left\{ -\sum\limits_{j=0}^{m}{
\dfrac{{C_{2j+1}}}{{\nu }^{2j+1}}}\right\}  \\ 
\times z^{1/2}\left\{ \tfrac{1}{2}i{\mathrm{Ai}\left( {\nu ^{2/3}\zeta }
\right) }H_{\nu }^{(1) }(\nu z) {-e^{\pi i/6}\mathrm{
Ai}_{-1}\left( {\nu ^{2/3}\zeta }\right) }J_{\nu }(\nu z)
\right\} .
\label{eq5.33}
\end{multline}
These are exact expressions, and can be used to compare numerically the coefficient functions with their approximations, and in particular the exact errors with our bounds (see below).

Next, we have from an application of \Cref{thm:main1}
\begin{multline}
\mathcal{A}_{2m+2}(\nu,z) \\ =\left( {\dfrac{z^{2}\zeta }{1-z^{2}
}}\right) ^{1/4}  \left[ {\exp \left\{ \sum\limits_{s=1}^{m}{\dfrac{\mathcal{\tilde{E}}
_{2s}(z) }{{\nu }^{2s}}}\right\} \cosh \left\{ 
\sum\limits_{s=0}^{m}{\dfrac{\mathcal{\tilde{E}}_{2s+1}(z) }{{
\nu }^{2s+1}}}\right\} }+\frac{1}{2}\tilde{\varepsilon}_{2m+2}(\nu,z)
\right] ,
\label{eq5.34}
\end{multline}
and 
\begin{multline}
\mathcal{B}_{2m+2}(\nu,z) =\dfrac{1}{{\nu }^{1/3}}\left\{ {\dfrac{z^{2}}{\zeta \left( 1-z^{2}\right) }}\right\} ^{1/4} \\ \times \left[ \exp \left\{\sum\limits_{s=1}^{m}{\dfrac{\mathcal{E}
_{2s}(z) }{{\nu }^{2s}}}\right\} \sinh \left\{ 
\sum\limits_{s=0}^{m}\dfrac{\mathcal{E}_{2s+1}(z) }{\nu 
^{2s+1}}\right\} +\frac{1}{2}\varepsilon _{2m+2}(\nu,z) 
\right] ,
\label{eq5.35}
\end{multline}
where $\mathcal{E}_{s}(z) $ and $\tilde{\mathcal{E}}
_{s}(z) $ are given by (\ref{eq40}) and (\ref{eq38}), and for $\nu >0$ and $z\in T_{0,-1}\cup T_{-1,0}$
\begin{multline}
\left\vert \tilde{\varepsilon}_{2m+2}(\nu,z) \right\vert \\ \leq \dfrac{1}{{\nu ^{2m+2}}}\exp \left\{ 
\sum\limits_{s=1}^{2m+1}{\dfrac{\mathrm{Re} \, \mathcal{\tilde{E}}_{s}(z) }{{
\nu }^{s}}}\right\} \tilde{e}_{2m+2,-1}(\nu,z) \left\{ 1+
\dfrac{\tilde{e}_{2m+2,-1}(\nu,z) }{2{{\nu }^{2m+2}}}\right\}
^{2} \\ 
+\dfrac{1}{{{\nu }^{2m+2}}}\exp \left\{ \sum\limits_{s=1}^{2m+1}{(-1)^{s}\dfrac{\mathrm{Re} \, \mathcal{\tilde{E}}_{s}(z) }{{\nu }^{s}}}
\right\} \tilde{e}_{2m+2,0}(\nu,z) \left\{ 1+\dfrac{\tilde{e}
_{2m+2,0}(\nu,z) }{2{{\nu }^{2m+2}}}\right\} ^{2},
\label{eq5.36}
\end{multline}
in which (for $j=0,-1$) 
\begin{multline}
\tilde{e}_{2m+2,j}(\nu,z) ={{\nu }^{2m+2}\delta }
_{2m+2,j}(\nu)  \\ 
+\omega _{2m+2,j}(\nu,z) \exp \left\{ {\nu }^{-1}\varpi
_{2m+2,j}(\nu,z) +{\nu }^{-2m-2}\omega _{2m+2,j}(\nu ,z)   \right\}  \\ 
+\tilde{\gamma}_{2m+2}(\nu,\xi) \exp \left\{ {{\nu }^{-1}
\tilde{\beta}_{2m+2}(\nu,\xi) +{\nu }^{-2m-2}\tilde{\gamma}
_{2m+2}(\nu,\xi) }\right\} .
\end{multline}
Here ${\delta }_{2m+2,0}(\nu) =0$, and from (\ref{exactdelta}), (\ref{ez0}) and (\ref{lambdaBessel})
\begin{equation}
{\delta }_{2m+2,\pm 1}(\nu) =\left( \frac{1}{2\pi }\right)
^{1/2}\frac{e^{{\nu }}\Gamma (\nu) }{{\nu }^{{\nu -(1/2)}}}
\exp \left\{ -\sum\limits_{j=0}^{m}\frac{C_{2j+1}}{{\nu }^{2j+1}}
\right\} -1;
\end{equation}
in addition
\begin{multline}
\omega _{2m+2,0}(\nu,z) =2\int_{0}^{z}{\left\vert \dfrac{{
\hat{F}_{2m+2}(t) }\left( 1-t^{2}\right) ^{1/2}}{t}{dt}
\right\vert } \\ 
+\sum\limits_{s=1}^{2m+1}\dfrac{1}{{\nu ^{s}}}\int_{0}^{z}{{\left\vert {
\sum\limits_{k=s}^{2m+1}}\dfrac{{{\hat{F}_{k}(t) \hat{F}
_{s+2m-k+1}(t) }}\left( 1-t^{2}\right) ^{1/2}}{t}{dt}
\right\vert }},
\label{eq5.39}
\end{multline}
\begin{equation}
\varpi _{2m+2,0}(\nu,z) =4\sum\limits_{s=0}^{2m}\frac{1}{{\nu
^{s}}}\int_{0}^{z}{{\left\vert \frac{{\hat{F}_{s+1}(t) }\left(
1-t^{2}\right) ^{1/2}}{t}{dt}\right\vert }},
\label{eq5.40}
\end{equation}
and $\omega _{2m+2,-1}(\nu,z) $ and $\varpi _{2m+2,-1}
(\nu,z) $ are the same except the paths of integration are from $z$ to infinity in the upper half plane. These can be taken as straight lines in both cases, in the latter case vertical lines from $z$ to infinity in the upper half plane.

Similarly 
\begin{multline}
\left\vert \varepsilon _{2m+2}(\nu,z) \right\vert \\ \leq \dfrac{
1}{{{\nu }^{2m+2}}}\exp \left\{ \sum\limits_{s=1}^{2m+1}\dfrac{\mathrm{Re} \,\mathcal{E}
_{s}(z) }{{\nu }^{s}}\right\} e_{2m+2,-1}
(\nu,z) \left\{ 1+\dfrac{e_{2m+2,-1}(\nu,z) }{2{{\nu }^{2m+2}}
}\right\} ^{2} \\ 
+\dfrac{1}{{\nu }^{2m+2}}\exp \left\{ \sum\limits_{s=1}^{2m+1}{(-1)^{s}\dfrac{\mathrm{Re} \,\mathcal{E}_{s}(z) }{{\nu }^{s}}}\right\}
e_{2m+2,0}(\nu,z) \left\{ 1+\dfrac{e_{2m+2,0}
(\nu,z) }{2{{\nu }^{2m+2}}}\right\} ^{2},
\label{eq5.41}
\end{multline}
where 
\begin{multline}
e_{2m+2,j}(\nu,z) ={{\nu }^{2m+2}\delta }_{2m+2,j}(\nu)  \\ 
+\omega _{2m+2,j}(\nu,z) \exp \left\{ {\nu }^{-1}\varpi
_{2m+2,j}(\nu,z) +{\nu }^{-2m-2}\omega _{2m+2,j}(\nu,z) \right\}  \\ 
+\gamma _{2m+2}(\nu,\xi) \exp \left\{ {\nu }^{-1}\beta
_{2m+2}(\nu,\xi) +{\nu }^{-2m-2}\gamma _{2m+2}
(\nu,\xi) \right\} .
\end{multline}

Before proceeding with numerical computations, let us illustrate how the above asymptotic solutions can be matched with the exact solutions. We do so we consider solutions recessive at $z=0$, with the other solutions done similarly. Now, by uniqueness of such solutions we immediately deduce from the $l=0$ solution of \Cref{thm:main1} that 
\begin{equation}
J_{\nu }(\nu z) =c_{m,0}(\nu) z^{-1/2}\left\{ 
\mathrm{Ai}\left( {\nu ^{2/3}\zeta }\right) \mathcal{A}_{2m+2}\left( {\nu ,z}
\right) +\mathrm{Ai}^{\prime }\left( {\nu ^{2/3}\zeta }\right) \mathcal{B}
_{2m+2}(\nu,z) \right\} ,
\label{5.43}
\end{equation}
for some constant $c_{m,0}(\nu) $. Letting $z\rightarrow 0$ in (\ref{eq5.34}) and (\ref{eq5.35}) and referring to (\ref{xiBessel}) and (\ref{xi0})
we have 
\begin{equation}
\mathcal{A}_{2m+2}(\nu,z) \sim z^{1/2}\zeta ^{1/4}\left[
\cosh \left\{ \sum\limits_{j=0}^{m}\dfrac{C_{2j+1}}{{\nu }^{2j+1}}
\right\} +\frac{1}{2}\tilde{\varepsilon}_{2m+2}(\nu,0)
\right] ,
\label{5.44}
\end{equation}
and 
\begin{equation}
\mathcal{B}_{2m+2}(\nu,z) \sim 
\frac{z^{1/2}}{\nu^{1/3} \zeta^{1/4}}
\left[ \sinh \left\{ \sum\limits_{j=0}^{m}\dfrac{
C_{2j+1}}{{\nu }^{2j+1}}\right\} +\frac{1}{2}\varepsilon _{2m+2}(\nu,0) \right] .
\label{5.45}
\end{equation}
Although we don't know $\tilde{\varepsilon}_{2m+2}(\nu,0) $ and $\varepsilon _{2m+2}(\nu,0) $ explicitly we can bound these values.
Specifically, from the above bounds we see that 
\begin{equation}
\left\vert \tilde{\varepsilon}_{2m+2}(\nu,0) \right\vert \leq 
\dfrac{1}{{{\nu }^{2m+2}}}\exp \left\{ \sum\limits_{j=0}^{m}\frac{C
_{2j+1}}{{\nu }^{2j+1}}\right\} \tilde{e}_{2m+2,-1}(\nu ,0)
\left\{ 1+\dfrac{\tilde{e}_{2m+2,-1}(\nu,0) }{2{{\nu }^{2m+2}}
}\right\} ^{2},
\end{equation}
since $\tilde{e}_{2m+2,0}(\nu,0) =0$, and in this 
\begin{multline}
\tilde{e}_{2m+2,-1}(\nu,0) ={{\nu }^{2m+2}\delta }
_{2m+2,-1}(\nu)  \\ 
+\omega _{2m+2,-1}(\nu,0) \exp \left\{ {{\nu }^{-1}\varpi
_{2m+2,-1}(\nu,0) +{\nu }^{-2m-2}\omega _{2m+2,-1}\left( {\nu
,0}\right) }\right\} .
\end{multline}

Similarly $\varepsilon _{2m+2}(\nu,0)$ satisfies the same bound, since $e_{2m+2,0}(\nu,0) =0$ and the analogously defined $e_{2m+2,-1}(\nu,0)$ is the same as $\tilde{e}_{2m+2,-1}(\nu,0)$.

On using (\ref{5.43}) - (\ref{5.45}), (\ref{Aiinfinity}) and (\ref{eq78a}) we arrive at 
\begin{multline}
\dfrac{\left( \frac{1}{2}{\nu }\right) ^{{\nu }}}{\Gamma \left( {\nu +1}
\right) }=\dfrac{c_{m,0}(\nu) \exp \left( -{\nu }\ln \left(
2\right) +{\nu }\right) }{2\pi ^{1/2}{\nu }^{1/6}} \\ 
\times \left[ \exp \left\{ -{\sum\limits_{j=0}^{m}}\dfrac{C_{2j+1}}{{\nu }
^{2j+1}}\right\} +\frac{1}{2}\tilde{\varepsilon}_{2m+2}
(\nu,0) -\frac{1}{2}\varepsilon_{2m+2}(\nu,0) \right] ,
\end{multline}
and therefore the desired value of the proportionality constant is given by
\begin{multline}
c_{m,0}(\nu) =\dfrac{2\pi ^{1/2}{\nu }^{{\nu -(5/6)
}}}{e^{\nu }\Gamma (\nu) } \\ 
\times \left[ {\exp \left\{ -{\sum\limits_{j=0}^{m}}\dfrac{C_{2j+1}}{{\nu }
^{2j+1}}\right\} }+\frac{1}{2}\tilde{\varepsilon}_{2m+2}(\nu,0) -\frac{1}{2}\varepsilon _{2m+2}(\nu,0) \right] ^{-1}.
\end{multline}

The identification of the Hankel functions can be done similarly. We omit details.

\subsection{ Numerical examples}

Examples of the performance of the error bounds given in (\ref{eq5.36}) and (\ref{eq5.41}) are shown in Figures \ref{fig:fig4} and \ref{fig:fig5}.

\begin{figure}[h!]
  \centering
  \includegraphics[scale=0.35]{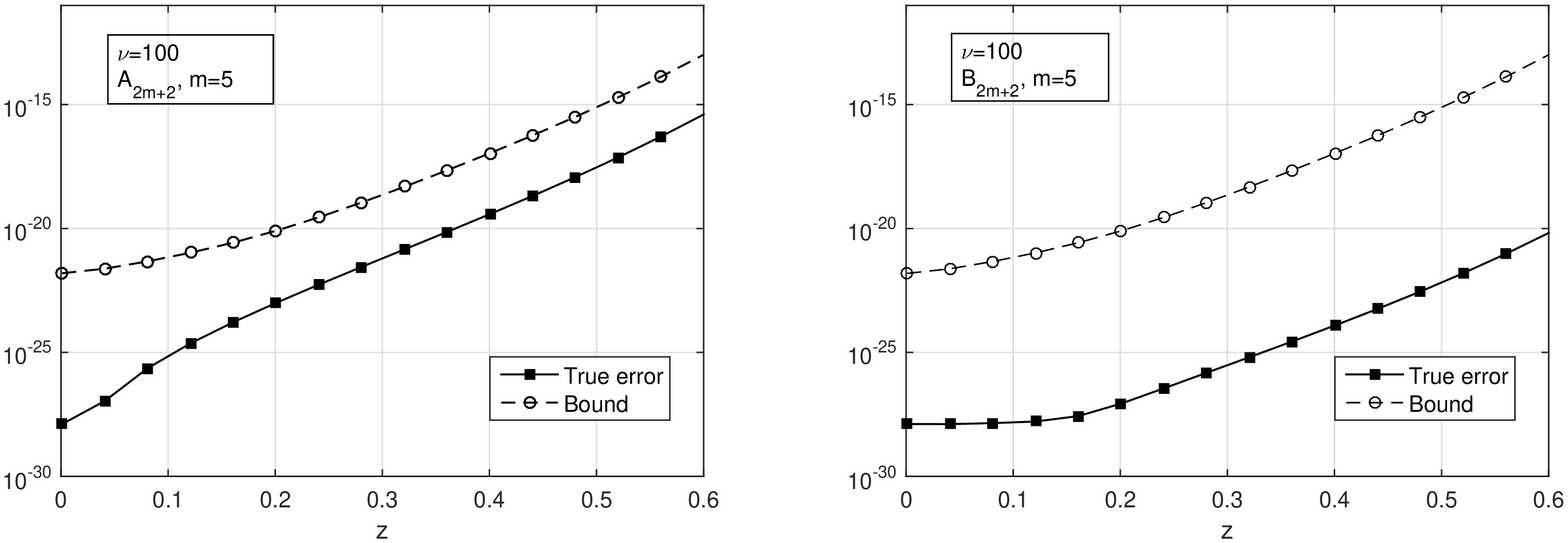}
  \caption{Comparison of the bounds given in (\ref{eq5.36}) and (\ref{eq5.41}) with the true numerical accuracy obtained when using 
(\ref{eq5.34}) and (\ref{eq5.35}) 
to approximate (\ref{eq5.32}) and (\ref{eq5.33}), respectively, for a fixed value of $m$ ($m=5$) and $\nu$ ($\nu=100$).}
  \label{fig:fig4}
\end{figure}

\begin{figure}[h]
  \centering
  \includegraphics[scale=0.35]{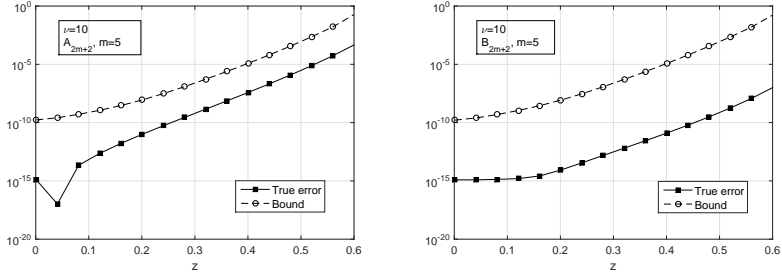}
  \caption{Comparison of the bounds given in (\ref{eq5.36}) and (\ref{eq5.41}) with the true numerical accuracy obtained 
when using (\ref{eq5.34}) and (\ref{eq5.35}) 
to approximate (\ref{eq5.32}) and (\ref{eq5.33}), respectively, for a fixed value of $m$ ($m=5$) and $\nu$ ($\nu=10$).}
  \label{fig:fig5}
\end{figure}

In the figures, these bounds are compared with the true numerical accuracy obtained when using (\ref{eq5.34}) and 
(\ref{eq5.35}) 
to approximate (\ref{eq5.32}) and (\ref{eq5.33}), respectively, for a fixed value of $m$ ($m=5$) and
two different values of $\nu$ ($\nu=10,\,100$). The computation of (\ref{eq5.32}) and (\ref{eq5.33}) is made using Maple with a large number of digits. 

For the bounds, two different types of numerical quadrature methods have been
considered to evaluate the integrals: (i) a Gauss-Legendre quadrature with 30 nodes for the integrals
in (\ref{eq5.39}) and (\ref{eq5.40}); (ii) an adaptative quadrature method over a truncated interval for the integrals for $\omega _{2m+2,-1}(\nu,z) $ and $\varpi _{2m+2,-1}(\nu,z) $. 

As can be seen in the figures, the bounds (\ref{eq5.36}) and (\ref{eq5.41}) track the exact errors quite well even
for moderate values of $\nu$. Also, the accuracy of the bound (\ref{eq5.36}) is better
than the accuracy of (\ref{eq5.41}), as expected (see \Cref{remark6}).  

\newpage

\appendix
\section{Exponential-type Liouville-Green expansions for Airy functions} 
\label{secA}
Let $\left\vert {\arg \left( {u^{2/3}\zeta }\right) }\right\vert
\leq {\frac{2}{3}}\pi $ (or equivalently, from (\ref{eq2}), $\left\vert \arg(u \xi) \right\vert \leq \pi $). Now $V=z^{1/4}\mathrm{Ai}
\left( {u^{2/3}\zeta }\right) $ satisfies 
\begin{equation}
\frac{d^{2}V}{d\xi ^{2}}=\left\{ u^{2}{-\frac{5}{36\xi ^{2}}}\right\} V.
\label{eq74}
\end{equation}
From \cite[Theorem 1.1]{Dunster:2020:LGE} where (1.14) we obtain a solution 
\begin{equation}
{V_{n}^{(0) }}(u,\xi) =\exp \left\{ 
\sum\limits_{s=1}^{n-1}{(-1) ^{s}}\frac{e_{s}(\xi) }{u^{s}}\right\} \left\{ {e^{-u\xi }+\varepsilon _{n}^{(0) }(u,\xi) }\right\},
\label{eq75}
\end{equation}
where from from (1.14) of that paper the coefficients are found to be
\begin{equation}
e_{s}(\xi) =\frac{a_{s}}{s\xi^{s}},
\label{eq80}
\end{equation}
with $a_{1}=a_{2}=\frac{5}{72}$ and subsequent terms satsfying the
recurrence relation (\ref{arec}). From this we apply \cite[(1.20)]
{Dunster:2020:LGE} to provide the bound 
\begin{multline}
\left\vert {\varepsilon _{n}^{(0) }(u,\xi) }
\right\vert \leq \left\vert u^{-n}{e^{-u\xi }}\right\vert \int_{\infty
}^{\xi }{\left\vert {\chi _{n}^{(0) }(u,t) dt}
\right\vert } \\ 
\times \exp \left\{ 4\int_{\infty }^{\xi }{\left\vert u^{-1}{T_{n}^{\left(
0\right) }(u,t) dt}\right\vert }+\int_{\infty }^{\xi }{
\left\vert u^{-n}{\chi _{n}^{(0) }(u,t) dt}
\right\vert }\right\},
\label{eq79}
\end{multline}
where 
\begin{equation}
\chi _{n}^{(0) }(u,\xi) =\frac{(-1)
^{n+1}}{\xi ^{n+1}}\left\{ 2a_{n}+\sum\limits_{s=1}^{n-1}\frac{
(-1)^{s+1}}{(u \xi)^{s}}\sum\limits_{k=s}^{n-1}{
a_{k}a_{s+n-k-1}}\right\},
\label{eq81}
\end{equation}
and
\begin{equation}
T_{n}^{(0) }(u,\xi) =\sum\limits_{s=0}^{n-2}(-1) ^{s}\frac{a_{s+1}}{u^{s}\xi ^{s+2}}.
\label{eq81a}
\end{equation}
In (\ref{eq79}) the paths for the three integrals is the ray
\begin{equation}
t=\tau \xi \ \left( {1\leq \tau <\infty }\right),
\label{eq86}
\end{equation}
if $\mathrm{Re}(u\xi) \geq 0$, and the ray 
\begin{equation}
t=\xi \mp i\xi \tau \ \left( {0\leq \tau <\infty }\right),
\label{eq87}
\end{equation}
if $\mathrm{Re}(u\xi) <0$ and $\pm \mathrm{Im}(u \xi) \geq 0$. These are chosen as the simplest paths satisfying the requirement that $\mathrm{Re}(ut)$ is increasing as $t$ runs along the path from $\xi $ to $\infty $,\ for each fixed $\xi $ in the cut plane $\left\vert \arg (u \xi) \right\vert \leq \pi $.
\begin{lemma}
For $p\geq 2$ 
\begin{equation}
\int_{\infty }^{\xi }{\left\vert \dfrac{dt}{t^{p}}\right\vert }=\dfrac{
\Lambda _{p}(u \xi) }{|\xi| ^{p-1}},
\label{eq85}
\end{equation}
where 
\begin{equation}
\Lambda _{p}(z) =\left\{ 
\begin{array}{ll}
\dfrac{1}{p-1}\  & \left( \mathrm{Re}\,z\geq 0\right) \\ 
\dfrac{\pi ^{1/2}\Gamma \left( \frac{1}{2}p-\frac{1}{2}\right) }{2\Gamma
\left( \frac{1}{2}p\right) } & \ \left( \mathrm{Re}\,z<0\right).
\end{array}
\right.
\label{eq91}
\end{equation}
\end{lemma}

\begin{proof}
If $\mathrm{Re}(u \xi) \geq 0$ we have from (\ref{eq86}) 
\begin{equation}
\int_{\infty }^{\xi }{\left\vert \frac{dt}{t^{p}}\right\vert }=\frac{1}{
|\xi| ^{p-1}}\int_{1}^{\infty }\frac{d\tau}{\tau ^{p}}
=\frac{1}{\left( {p-1}\right) |\xi| ^{p-1}},
\label{eq88}
\end{equation}
and for $\mathrm{Re}(u \xi) <0$, using (\ref{eq87}), 
\begin{equation}
\int_{\infty }^{\xi }{\left\vert \frac{dt}{t^{p}}\right\vert }=\frac{1}{
|\xi| ^{p-1}}\int_{0}^{\infty }\frac{d\tau }{\left( 
\tau ^{2}+1\right) ^{p/2}}=\frac{\pi ^{1/2}\Gamma \left( {\frac{1}{2}}p-
\frac{1}{2}\right) }{2\Gamma \left( {\frac{1}{2}}p\right) \left\vert \xi
\right\vert ^{p-1}}.
\label{eq89}
\end{equation}
Combining these two gives the result.
\end{proof}

Then from the triangle inequality applied to (\ref{eq81}) and using (\ref
{eq85}) 
\begin{multline}
\int_{\infty }^{\xi }{\left\vert \chi _{n}^{(0)}(u,t) dt\right\vert }\leq \frac{2a_{n}\Lambda _{n+1}
(u\xi) }{|\xi| ^{n}} \\ +\frac{1}{\left\vert \xi
\right\vert ^{n+1}}\sum\limits_{s=0}^{n-2}{\frac{\Lambda _{n+s+2}(u\xi) }{\left\vert u\xi \right\vert ^{s}}\sum\limits_{k=s+1}^{n-1}{a_{k}a_{s+n-k}}},
\label{eq82}
\end{multline}
and 
\begin{equation}
\int_{\infty }^{\xi }{\left\vert {T_{n}^{(0) }(u,t)
dt}\right\vert }\leq \frac{1}{|\xi| }
\sum\limits_{s=0}^{n-2}\frac{a_{s+1}\Lambda _{s+2}(u \xi) }{
\left\vert u\xi \right\vert^{s}}.
\label{eq83}
\end{equation}

Now for $\mathrm{Re}\,z<0$ we have from (\ref{eq91}) and Stirling's formula \cite[Chap. 3, Eq. (8.16)]{Olver:1997:ASF} that $\Lambda _{p}(z) \sim \left\{ \pi
/\left( 2p\right) \right\} ^{1/2}$ as $p\rightarrow \infty $ , which suggests the
simplification
\begin{equation}
\Lambda _{p}(z) \leq \Lambda _{p}:=\dfrac{\pi ^{1/2}\Gamma
\left( \frac{1}{2}p-\frac{1}{2}\right) }{2\Gamma \left( \frac{1}{2}p\right) }
\ \left( p\geq 2\right),
\label{lambound}
\end{equation}
for all $z$. The following lemma establishes this to be true.

\begin{lemma}
For $p\geq 2$
\begin{equation}
\dfrac{1}{p-1}<\dfrac{\pi ^{1/2}\Gamma \left( \frac{1}{2}p-\frac{1}{2}
\right) }{2\Gamma \left( \frac{1}{2}p\right) }.
\label{G1}
\end{equation}
\end{lemma}

\begin{proof}
\begin{equation}
\dfrac{2\Gamma \left( \frac{1}{2}p\right) }{\pi ^{1/2}\left( p-1\right)
\Gamma \left( \frac{1}{2}p-\frac{1}{2}\right) }=\dfrac{\Gamma \left( \frac{1}{2}p\right) \Gamma \left( \frac{1}{2}\right) }{\pi \Gamma \left( \frac{1}{2}p+\frac{1}{2}\right) }=\frac{1}{\pi }B\left( \tfrac{1}{2}p,\tfrac{1}{2}
\right),
\label{G2}
\end{equation}
where $B\left( p,q\right) $ is the Beta function \cite[Chap. 2, Sect. 1.6]
{Olver:1997:ASF} 
\begin{equation}
B\left( p,q\right) =\dfrac{\Gamma \left( p\right) \Gamma \left( q\right) }{
\Gamma \left( p+q\right) }=\int_{0}^{1}v^{p-1}\left( 1-v\right) ^{q-1}dv\
\left( \mathrm{Re}\,p>0,\ \mathrm{Re}\,q>0\right).
\label{beta}
\end{equation}
From its integral representation we see that this function is monotonically
decreasing as a function of (positive) $p$ for each fixed $q>0$. Therefore for $p\geq 2$
\begin{equation}
\dfrac{2\Gamma \left( \frac{1}{2}p\right) }{\pi ^{1/2}\left( p-1\right)
\Gamma \left( \frac{1}{2}p-\frac{1}{2}\right) }\leq \frac{1}{\pi }B\left( 1,
\tfrac{1}{2}\right) =\dfrac{\Gamma \left( \tfrac{1}{2}\right) }{\pi \Gamma
\left( \tfrac{3}{2}\right) }=\frac{2}{\pi }<1,
\label{G3}
\end{equation}
and the stated result follows.
\end{proof}

Now by uniqueness of recessive solutions $\mathrm{Ai}\left( {u^{2/3}\zeta }
\right) =c(u)\zeta ^{-1/4}{V_{n}^{(0) }}(u,\xi) $.
Then using
\begin{equation}
\mathrm{Ai}\left( {u^{2/3}\zeta }\right) \sim \frac{e^{-u\xi }}{2\pi
^{1/2}u^{1/6}\zeta ^{1/4}}\ \left( u\xi \rightarrow +\infty \right),
\label{Aiinfinity}
\end{equation}
we find $c(u)=\frac{1}{2}u^{-1/6}\pi ^{-1/2}$, and as a result (\ref{eq24}) follows
from (\ref{eq75}) with
\begin{equation}
\eta _{n}^{(0) }(u,\xi) 
:=e^{u\xi }\varepsilon _{n}^{(0) }(u,\xi).
\end{equation}
The bound (\ref{eq26}) on this error term follows from (\ref{eq79}), (\ref{eq82}), (\ref{eq83}) and (\ref{lambound}).

For the derivative of the Airy function we note that $\tilde{V}={\zeta }^{-1/4}\mathrm{Ai}^{\prime }\left( {u^{2/3}\zeta }\right) $ satisfies 
\begin{equation}
\frac{d^{2}\tilde{V}}{d\xi ^{2}}=\left\{ u^{2}+\frac{7}{36\xi ^{2}}
\right\} \tilde{V}.
\label{eq78}
\end{equation}
Then using 
\begin{equation}
\mathrm{Ai}^{\prime }\left( u^{2/3}\zeta \right) \sim -\frac{u^{1/6}{\zeta }
^{1/4}e^{-u\xi }}{2\pi ^{1/2}}\ \left( u\xi \rightarrow +\infty \right),
\label{eq78a}
\end{equation}
we obtain in a similar manner
\begin{equation}
\mathrm{Ai}^{\prime }\left( {u^{2/3}\zeta }\right) =-\frac{u^{1/6}{\zeta }^{1/4}}{2\pi ^{1/2}}\exp \left\{\sum\limits_{s=1}^{n-1}{(-1)
^{s}\frac{\tilde{a}_{s}}{su^{s}\xi ^{s}}}\right\} \left\{ {e^{-u\xi }+\tilde{\varepsilon}_{n}^{(0) }(u,\xi) }\right\},
\label{Aip}
\end{equation}
where $\tilde{a}_{1}=\tilde{a}_{2}=-\frac{7}{72}$ and subsequent
terms given by (\ref{arec}), and $\tilde{\varepsilon}_{n}^{(0)}(u,\xi)$ satisfies the same bounds (\ref{eq79}), (\ref{eq82}) and (\ref{eq83}) as $\varepsilon _{n}^{(0) }(u,\xi)$ but with $a_s$ replaced by $\left \vert \tilde{a}_s \right \vert$ throughout. This establishes
(\ref{eq25}).

\section*{Acknowledgments}
Financial support from Ministerio de Ciencia e Innovaci\'on, Spain, 
projects MTM2015-67142-P (MINECO/FEDER, UE) and PGC2018-098279-B-I00 (MCIU/AEI/FEDER, UE)
 is acknowledged. 

\bibliographystyle{siamplain}
\bibliography{biblio}
\end{document}